\documentclass{amsart}
\usepackage{amssymb}
\usepackage{amsmath}
\usepackage{amsthm}
\usepackage{paralist}
\usepackage[misc]{ifsym}
\usepackage{epsfig} 
\usepackage{epstopdf} 
\usepackage[hidelinks]{hyperref}
\allowdisplaybreaks
\usepackage[all,cmtip]{xy}
\usepackage{xcolor}
\usepackage{mathrsfs}
\usepackage{comment}

\newtheorem{theorem}{Theorem}[section]
\newtheorem{corollary}[theorem]{Corollary}

\newtheorem{lemma}[theorem]{Lemma}
\newtheorem{proposition}[theorem]{Proposition}

\theoremstyle{definition}
\newtheorem{definition}[theorem]{Definition}
\newtheorem{remark}[theorem]{Remark}

\newcommand{\lra}{\longrightarrow}
\newcommand{\rr}{\mathbb{R}}
\newcommand{\ff}{\mathbb F}
\newcommand{\frakA}{\mathfrak A}
\newcommand{\frakg}{\mathfrak g}

\newcommand{\frakU}{\mathfrak U}

\newcommand{\Lie}{\operatorname{Lie}}

\newcommand{\Hor}{\operatorname{Hor}}
\newcommand{\calA}{\mathcal A}

\newcommand{\calD}{\mathcal D}

\newcommand{\calJ}{\mathcal J}

\newcommand{\calP}{\mathcal P}
\newcommand{\calV}{\mathcal V}

\newcommand{\subs}{\subseteq}
\newcommand{\pr}{\operatorname{pr}}
\newcommand{\Lra}{\Longrightarrow}

\newcommand{\id}{\operatorname{id}}

\newcommand{\ti}[1]{\widetilde{{#1}}}
\newcommand{\baseRing}[1]{\ensuremath{\mathbb{#1}}}
\newcommand{\R}{\baseRing{R}}

\newcommand{\jdef}[1]{\emph{#1}}

\newcommand{\CD}{\ensuremath{{\calA_d}}}

\newcommand{\HLc}[1]{\ensuremath{{h^{{#1}}}}}
\newcommand{\SG}{\ensuremath{G}}

\usepackage[
backend=biber,
style=alphabetic,
sorting=nyt,
uniquename=false,
giveninits=true,
doi=false,url=false
]{biblatex}
\renewbibmacro{in:}{} 
\begin{document}

\title[Lagrangian reduction of symmetric discrete mechanical systems]{Lagrangian reduction of symmetric discrete mechanical systems: a survey}

\author{Mat\'ias I. Caruso}
\address{\textnormal{(M. I. Caruso)} Depto. de Matem\'atica, Instituto Balseiro \\ Universidad Nacional de Cuyo - C.N.E.A. \\ Av. Bustillo 9500 \\ San Carlos de Bariloche \\ R8402AGP \\ Argentina \hfill \break
	\indent CONICET}
\email{matias.caruso@ib.edu.ar}

\author{Javier Fern\'andez}
\address{\textnormal{(J. Fern\'andez)} Depto. de Matem\'atica, Instituto Balseiro \\ Universidad Nacional de Cuyo - C.N.E.A. \\ Av. Bustillo 9500 \\ San Carlos de Bariloche \\ R8402AGP \\ Argentina}
\email{jfernand@ib.edu.ar}

\author{Cora Tori}
\address{\textnormal{(C. Tori)} Depto. de Ciencias B\'asicas \\ Facultad de Ingenier\'ia \\ Universidad Nacional de La Plata.
	Calle 116 entre 47 y 48 \\ La Plata \\ Buenos Aires \\ 1900 \\ Argentina \hfill \break
	\indent Centro de Matem\'atica de La Plata (CMaLP)}
\email{cora.tori@ing.unlp.edu.ar}

\author{Marcela Zuccalli}
\address{\textnormal{(M. Zuccalli)} Depto. de Matem\'atica \\ Facultad de Ciencias Exactas \\ Universidad Nacional de La Plata.
	Calles 50 y 115 \\ La Plata \\ Buenos Aires \\ 1900 \\ Argentina \hfill \break
	\indent Centro de Matem\'atica de La Plata (CMaLP)}
\email{marce@mate.unlp.edu.ar}


\begin{abstract}
In this note we survey some of our results on the Lagrangian reduction of discrete-time mechanical systems (DMSs). It is intended as an introduction to the general ideas that we used in the reduction of DMSs with nonholonomic constraints, DMSs with external forcing, as well as a theory of reduction by stages for such systems. This line of work was inspired by the paper and the monograph written by H. Cendra, J. Marsden and T. Ratiu in 2001.
\end{abstract}

\maketitle

\section{Introduction}

In 2001 H. Cendra, J. Marsden and T. Ratiu published a paper and a
monograph~\cite{ar:cendra_marsden_ratiu:2001:geometric_mechanics_lagrangian_reduction_and_nonholonomic_systems,ar:cendra_marsden_ratiu:2001:lagrangian_reduction_by_stages}
discussing several aspects of the Lagrangian reduction of mechanical
systems that, first, survey the current state of affairs and, then,
extend the techniques to other areas, like nonholonomic mechanical
systems and reduction by stages. The purpose of this paper, inspired
by the approach of that work, is to survey some of our work that
explores the same areas but for discrete-time mechanical systems. It
is our humble tribute to our teacher, colleague and friend, Hern\'an
Cendra.

The study of symmetries is, arguably, as old as humankind. In
Classical Mechanics, the work of historic figures like Euler, Jacobi,
Lagrange, Hamilton, Routh and Poincar\'e, among others, has explored
the use of symmetries to simplify the description of a given system by
eliminating redundant degrees of freedom: this usually has the benefit
of leading to a system that is easier to describe ---the reduced
system--- and, also, pointing towards a more essential part of the
dynamics. As we know, Classical Mechanics can be described in
different, essentially equivalent, formalisms: Lagrangian, Hamiltonian
and, more recently, Dirac. Depending on the tools used to determine
the dynamics of the system, we talk, for example, about variational
principles or symplectic
formalism. In~\cite{ar:cendra_marsden_ratiu:2001:geometric_mechanics_lagrangian_reduction_and_nonholonomic_systems}
and~\cite{ar:cendra_marsden_ratiu:2001:lagrangian_reduction_by_stages},
they focus on the Lagrangian--variational description, and so shall we
in what follows, except that our time variable will be discrete.

Discrete mechanical systems have been considered for a number of
decades already and one of the main thrusts for their study is that
they can be used to derive numerical integrators for their
continuous-time counterparts. Also, the study of the dynamical
properties of the discrete systems immediately gives information about
the integrator and, when available, conservation properties of the
discrete systems impose strict limitations on the integrator that, in
many cases, lead to well behaved algorithms, especially when used for
long time scales.

Still, our purpose is not to study the relation between the continuous
and discrete systems from the perspective of the numerical integrators
---i.e., the error analysis--- but, rather, the reduction process of
discrete systems per se. The basic setting, discussed in Section 2, is
as follows: we are given a discrete mechanical system $(Q,L_d)$ whose
trajectories $q_\cdot = (q_0,\ldots,q_N) \in Q^{N+1}$ are determined
by a criticality condition of a functional $S_d$ under certain
infinitesimal variations $\delta q_\cdot$. This criticality condition
can be equivalently described via the (discrete) equations of motion,
that are algebraic rather than differential. When a Lie group $G$ acts
on the left on $Q$ in such a way that $\pi:Q\rightarrow Q/G$ is a
principal $G$-bundle and (the tangent lift of) this action leaves
$L_d$ invariant, we say that $G$ is a symmetry group of $(Q,L_d)$. We
can see that the criticality condition on $q_\cdot$ leads to another
criticality condition on discrete paths $(v_\cdot,\tau_\cdot)$ in the
reduced space $\ti{G} := (Q\times\SG)/G$, subject to an appropriate
set of infinitesimal variations $(\delta v_\cdot, \delta
\tau_\cdot)$. The main result is that the two criticality conditions
are, in fact, equivalent. Also, each such condition is described in
terms of some equations of motion, leading to what are known as ``$4$
point theorems''. This analysis goes full circle by a reconstruction
result that states that all reduced paths $(v_\cdot,\tau_\cdot)$ can
be lifted to paths $q_\cdot$ in such a way that the former are
critical if and only if the latter are critical. An interesting
feature of the equations of motion of the reduced system is that,
using a principal connection on $\pi$, they split into two sets: the
horizontal and the vertical equations. The second set of equations is
intimately related to the momentum conservation of the system on $Q$
due to the $G$ symmetry ---a discrete version of Noether's Theorem.
As an application of these ideas, we discuss the discrete version of
Routh reduction that was first presented
by~\cite{ar:jalnapurkar_leok_marsden_west:2006:discrete_routh_reduction} in
an ad-hoc manner and, recently,
in~\cite{ar:caruso_fernandez_tori_zuccalli:2026:remarks_on_structures_and_preservation_in_forced_discrete_mechanical_systems_of_Routh_type},
as part of this general Lagrangian reduction procedure for discrete
mechanical systems.

The construction described so far can be extended in various
directions, always keeping the same basic structure, but with
additional data, according to each setting. In Section 3 we discuss
the reduction theory in the case of discrete mechanical systems that
are subjected to external forces. This subject was treated
in~\cite{ar:caruso_fernandez_tori_zuccalli:2023:lagrangian_reduction_of_forced_discrete_mechanical_systems}
where we proved a $4$ point theorem and a reconstruction theorem. In
addition, we saw that in some particular cases, the reduced system
preserves a certain non-canonical symplectic structure.

Another extension of the basic discrete reduction process is to
consider discrete mechanical systems that are subject to nonholonomic
constraints. The discrete systems that we consider have variational
and kinematic constraints that are not related; thus, they are the
discrete analogue of the generalized nonholonomic systems considered
in~\cite{ar:cendra_grillo:2006:generalized_nonholonomic_mechanics_servomechanisms_and_related_brackets}. This
reduction case was considered
in~\cite{ar:fernandez_tori_zuccalli:2010:lagrangian_reduction_of_discrete_mechanical_systems}
and we will revisit it in Section $4$. Once again, we obtain a $4$
point theorem as well as a reconstruction result. There are many other
features in this setting that clearly parallel similar properties of
the continuous systems.

A not so desirable characteristic of the discrete reduction processes
described so far is that even though the starting data corresponds to
a discrete mechanical system, the discrete reduced system produced is
not a mechanical system but, rather, a more general dynamical
system. This phenomenon afflicts the continuous systems too and has
been extensively treated
in~\cite{ar:cendra_marsden_ratiu:2001:lagrangian_reduction_by_stages}: the
idea is to define a class of dynamical systems that is larger than
that of (continuous) mechanical systems and that also contains the
systems obtained by reduction of these systems. Also, a general
reduction process is defined for the larger class of systems (that
restricts to the already known reduction for the mechanical
ones). Then they prove that, under certain conditions, the reduction
of a system by a symmetry group $G$ can be realized by stages, that
is, if $H\subset G$ is a normal subgroup, then one can reduce the
system, first by $H$ and, then, by the remaining symmetries; the
alternative, full reduction by $G$ leads to a reduced system that is
equivalent to the one obtained reducing by stages. We explored
in~\cite{ar:fernandez_tori_zuccalli:2016:lagrangian_reduction_of_discrete_mechanical_systems_by_stages}
and~\cite{ar:fernandez_tori_zuccalli:2020:lagrangian_reduction_of_nonholonomic_discrete_mechanical_systems_by_stages}
how the same problem and solution happens in the discrete case. In
Section 5 we briefly touch on some of the main ideas in this very
technical subject.

A technical point here is that the passage from the description of the
system on $Q$ to a system on $\ti{G}$ requires the usage of a
geometric construction called affine discrete connection; these
connections as well as the corresponding isomorphisms will be
described in Section 2.5. In the continuous case, this is achieved
using a principal connection on $\pi$. We should warn the reader that
in this paper we are strengthening the original definition of affine
discrete connection given
in~\cite{ar:fernandez_tori_zuccalli:2010:lagrangian_reduction_of_discrete_mechanical_systems}
to make it compatible with other, later, studies on the subject of
discrete connections on a principal bundle
(see~\cite{ar:fernandez_zuccalli:2013:a_geometric_approach_to_discrete_connections_on_principal_bundles}).

As we said earlier, this paper discusses our variational approach to
the reduction of discrete mechanical systems. Important alternative
treatments that we don't cover here include the elegant and powerful
groupoid approach
(see~\cite{ar:marrero_martin_martinez-discrete_lagrangian_and_hamiltonian_mechanics_on_lie_groupoids,ar:marrero_martin_martinez-discrete_lagrangian_and_hamiltonian_mechanics_on_lie_groupoids-corrigendum})
and the more recent one for discrete Dirac systems
(see~\cite{ar:rodriguezAbella_leok-discrete_dirac_reduction_of_implicit_lagrangian_systems_with_abelian_symmetry_groups}).

\section{Reduction of symmetric discrete mechanical systems}

In this section we consider discrete mechanical systems with symmetries and a process of reduction using the notion of affine discrete connection. An interesting particular case of these ideas is the so-called discrete Routh reduction.

\subsection{Some preliminaries on product manifolds}\label{subsection:preliminaries_on_product_manifolds}

In order to describe with precision the elements that appear when dealing with discrete mechanical systems, we begin with a few general remarks on product manifolds.

Given a smooth manifold $Q$, consider the canonical projections onto the first and second factor $\pr_i : Q \times Q \lra Q$, $i = 1,2$. Using the product structure of the manifold $Q \times Q$, we have the following identification for its tangent bundle:
\[
T(Q \times Q) \simeq  \pr_1^{\ast} (TQ) \oplus \pr_{2}^{\ast} (TQ),
\]
where $\pr_{i}^{\ast} (TQ)$ denotes the pullback of the tangent bundle $\tau_{Q}: TQ \lra Q$ over $Q \times Q$ by $\pr_i$, for $i=1,2$. If we define $j_1 : \pr_{1}^{\ast} (TQ) \lra T(Q \times Q)$ by $j_1(\delta q) := (\delta q,0)$, we obtain an isomorphism of vector bundles
\[
\pr_{1}^{\ast}  (TQ) \simeq TQ^{-} := \ker(T \pr_2) \subs T(Q \times Q),
\]
where $T\pr_2 : T(Q \times Q) \lra TQ$ is the tangent map of $\pr_2$.

Similarly, defining $j_2 : p_{2}^{*}  (TQ) \lra T(Q \times Q)$ by $j_2(\delta q) := (0,\delta q)$, we obtain
\[
\pr_{2}^{*} (TQ) \simeq TQ^{+} := \ker(T \pr_1) \subs T(Q \times Q).
\]

Thus, we have the decomposition $\displaystyle{T(Q \times  Q) = TQ^{-} \oplus TQ^{+}}$ that, in turn, yields the decomposition
$\displaystyle{T^{\ast}(Q \times Q) = (TQ^{-})^{\circ} \oplus (TQ^{+})^{\circ}}$ and the natural identifications
$(TQ^{+})^{\circ} \simeq (TQ^{-})^{*} \simeq \pr_1^{*}(T^{*}Q)$ and 
$(TQ^{-})^{\circ} \simeq (TQ^{+})^{*} \simeq \pr_2^{*}(T^{*}Q)$, where $\pr_i^{*}(T^{*}Q)$ is the pullback bundle of the cotangent bundle $\pi_{Q}:T^*Q \lra Q$ over $Q\times Q$ by $\pr_i$, with $i=1,2$ and the superscript $\circ$ denotes the corresponding annihilator.

For every smooth function $\displaystyle{f : Q \times Q \lra X}$, where $X$ is a smooth manifold, we define $D_{1}f := Tf \circ j_1$ and $D_{2}f := Tf \circ j_2$, where $Tf:T(Q \times Q) \lra TX$ denotes the tangent map of $f$. Hence,
\[
Tf(q_0,q_1)(\delta q_0,\delta q_1) = D_{1}f(q_0,q_1)(\delta q_0) + D_{2}f(q_0,q_1)(\delta q_1).
\]

If $f : Q \times Q \lra \rr$, then $D_{1}f(q_0,q_1) \in T^{*}_{q_0}Q$ and $\displaystyle{D_{2}f(q_0,q_1) \in T^{*}_{q_1}Q}$.

\subsection{Discrete mechanical systems (DMS) and their dynamics}

We briefly recall the notion of discrete mechanical system and its dynamics.

\begin{definition}\label{def:dms}
A \emph{discrete mechanical system} (DMS) consists of a pair $(Q,L_d)$, where $Q$ is a smooth manifold of dimension $n$, the \emph{configuration space}, and $L_d : Q \times Q \lra \rr$ is a smooth function, the \emph{discrete Lagrangian}.
\end{definition}

\begin{definition}
A \emph{discrete curve} on a smooth manifold $Q$ is a smooth function $q. : \{ 0,\ldots,N \} \lra Q$ that is usually identified with its image $(q_0,q_1,...,q_N)$, which is an element of the product manifold $Q^{N+1}$, where $q_k := q_{.}(k)$, for all $k\in \{ 0,\ldots,N \}$.
\end{definition}

Given the space of discrete curves $\mathcal{C}_{d}\left( Q \right) := \left\{ q.:\left\{ 0,\ldots,N \right\} \lra Q\right\} $, its tangent space at $q.$ is described as
\[
T_{q.}\left( \mathcal{C}_{d}\left( Q\right)\right) =\left\{
v_{q.}:\left\{ 0,\ldots,N \right\} \lra TQ  \ \mbox{ such that} \ \tau_{Q}\circ v_{q.}=q.\right\}.
\]

\begin{definition}
An \emph{infinitesimal variation} of a discrete curve $q.$ is a function $\delta q. : \{ 0,\ldots,N \} \lra TQ$ such that $\delta q_k \in T_{q_k}Q \ \forall\, k=0,...,N $. It is said that an infinitesimal variation has \emph{fixed endpoints} if $\delta q_0 = 0$ and $\delta q_N = 0$.
\end{definition}

\begin{definition}
The \emph{discrete action} of the DMS $(Q,L_d)$ is the smooth function $S_d:\mathcal{C}_{d}\left( Q \right) \lra \rr$ given by
\[
S_d(q.) := \sum_{k=0}^{N-1} L_d(q_k,q_{k+1}).
\]
\end{definition}

The dynamics of a DMS is determined by the discrete Hamilton principle, that establishes that the trajectories of a DMS are the extremal curves of $S_d$, considering infinitesimal variations with fixed endpoints.

\begin{definition}
A discrete curve $q.$ is a \emph{trajectory} of the DMS $(Q,L_d)$ if it satisfies
$\displaystyle{dS_d(q.)(\delta q.)= 0}$
for all infinitesimal variations $\delta q.$ of $q.$ with fixed endpoints.
\end{definition}

It is well--known (see \cite[Section 1.3.1]{ar:marsden_west:2001:discrete_mechanics_and_variational_integrators}) that the trajectories of a DMS may be characterized as the solution of a set of algebraic equations.

\begin{theorem}\label{DEE-L}
A discrete curve $q.$ is a trajectory of the DMS $(Q,L_d)$ if and only if $q.$ satisfies
\begin{equation}\label{dELe}
D_2 L_d(q_{k-1},q_k) + D_1 L_d(q_k,q_{k+1})  = 0 \in T_{q_k}^{*}Q \ \text{for all } k=1,\ldots,N-1,
\end{equation}
which are known as \emph{discrete Euler--Lagrange equations}.
\end{theorem}

\subsection{Discrete Legendre transforms and regularity}

Just like their continuous counterparts, the equations of motion of a DMS may admit no solution. We now give certain conditions that guarantee the existence of solutions.
 
\begin{definition}
Given a DMS $(Q,L_d)$, the maps 
$\ff^+L_{d}, \ff^-L_{d} : Q \times Q \lra T^*Q$ defined by
\[
\ff^+L_{d}(q_0,q_1) := D_2 L_d(q_0,q_1) \ \ \ \mbox{and}\ \ \ \ff^-L_{d}(q_0,q_1) := -D_1 L_d(q_0,q_1)
\]
are called \emph{discrete Legendre transforms}.
\end{definition}

\begin{remark}\label{preservation}
The maps $\ff^+L_{d}$ and $\ff^-L_{d}$ preserve the base points of the fiber bundles $\pr_{i}:Q\times Q \lra Q$ and $\pi_{Q}:T^*Q \lra Q$ with $i=1,2$, respectively. That is, 
$\ff^-L_{d}(q_0,q_1) \in T^{*}_{q_0}Q$ and $\ff^+L_{d}(q_0,q_1) \in T^{*}_{q_1}Q$. Hence, they define sections of the fiber bundles $\pr_{i}^{*} (T^{*}Q)$, with $i=1,2$; that is, $\ff^-L_{d} \in \Gamma(Q \times Q, \pr_{1}^{*} (T^{*}Q))$ and $\ff^+L_{d} \in \Gamma(Q \times Q, \pr_{2}^{*} (T^{*}Q)).$
\end{remark} 
 
\begin{remark}
It is worth noting that the discrete Euler--Lagrange equations \eqref{dELe} are equivalent to
$\displaystyle{\ff^+L_{d}(q_{k-1},q_{k}) =\ff^-L_{d}(q_{k},q_{k+1})}$ for all $k=1,...,N-1$.
\end{remark}

\begin{definition}
A DMS $(Q,L_d)$ is said to be \emph{regular} if $\ff^+L_{d}$ and $\ff^-L_{d}$ are local isomorphisms of fiber bundles or, equivalently, if they are local diffeomorphisms. That is, if the maps $\displaystyle{\phi_{q_1}:Q\lra T^{*}_{q_1}Q}$ and $\displaystyle{\phi_{q_0}:Q\lra T^{*}_{q_0}Q}$, given by
\[
\begin{split}
\phi_{q_1}(q_0) &:= \ff^+L_{d}(q_0,q_1)=D_2 L_d(q_0,q_1), \\
\phi_{q_0}(q_1) &:= \ff^-L_{d}(q_0,q_1)=-D_1 L_d(q_0,q_1)
\end{split}
\]
are local diffeomorphisms for every $q_0, q_1 \in Q$.

In some special cases, it may occur that both discrete Legrendre transforms are global diffeomorphisms. In this case, $(Q,L_d)$ is said to be \emph{hyperregular}.
\end{definition}

As mentioned before, the flow of a DMS is not always well defined, but the regularity of the system (together with a given trajectory) can be used to guarantee its existence (locally, at least). To state this more precisely, we prove the following result, noting that it is enough to consider curves of length $N=2$ (see the result in \cite{ar:borda_fernandez_grillo:2013:discrete_second_order_constrained_lagrangian_systems_first_results} mentioned in the following section, considering that the forces are zero).

\begin{theorem}\label{equations}
Let $(Q,L_d)$ be a regular DMS. Given a trajectory $(q_0,q_1,q_2)$ of $(Q,L_d)$, there exist open subsets $U, V \subs Q \times Q$ and a diffeomorphism ${\bf F}_{L_d} : U \lra V$ such that
\begin{enumerate}
\item $(q_0,q_1) \in U, (q_1,q_2) \in V$ and ${\bf F}_{L_d}(q_0, q_1) = (q_1,q_2)$.
\item For every $(\tilde{q_0},\tilde{q_1}) \in U$ if $\tilde{q_2} := \pr_2({\bf F}_{L_d}(q_0, q_1))$, then $(\tilde{q_0}, \tilde{q_1},  \tilde{q_2})$ is a trajectory of $(Q,L_d)$.
\item Every trajectory $(q'_0, q'_1, q'_2 )$ of $(Q,L_d)$ such that $(q'_0, q'_1) \in U$ and $(q'_1, q'_2) \in V$ satisfies
$(q'_1, q'_2) = {\bf F}_{L_d}(q'_0, q'_1)$.
\end{enumerate}
\end{theorem}

\subsection{Symmetric Discrete Mechanical Systems}

We now recall the notion of symmetry of a discrete mechanical system and a well--known preservation result.

Given a DMS $(Q,L_d)$, let us consider a Lie group $G$ acting on the left on $Q$ in a free and proper way by an action denoted $l^Q$, defining a principal bundle $\pi:Q\lra Q/G$ with structure group $G$. From now on, we will always deal with actions of Lie groups satisfying these conditions.

Consider the diagonal action of $G$ on $Q\times
Q$ defined by
\[
l^{Q\times Q}_g(q_0,q_1):=(l^Q_g(q_0),l^Q_g(q_1))
\]
and the lifted action of $G$ on $TQ$ given by $l^{TQ}_{g}(v_q):= T_ql^Q_{g}(v_q)$, where $v_q \in T_{q}Q$.

\begin{definition}
A Lie group $G$ is a \emph{symmetry group} of the DMS $(Q,L_d)$ if $L_d$ is $G$-invariant. That is, if
\[
L_d(l^Q_g(q_0),l^Q_g(q_1)) = L_d(q_0,q_1)
\]
for every $g \in G$ and every $(q_0,q_1)\in Q\times Q$.

In this case, it is possible to define a map on the quotient manifold $l_d :( Q \times Q)/G \lra \rr$ given by $l_d([q_0,q_1]_G):=L_d(q_0,q_1)$ and called \emph{reduced discrete Lagrangian}.
\end{definition}

In the discrete setting, there exists a conservation law in presence of symmetries in perfect analogy with the one existing for continuous systems.

\begin{definition}
Given a symmetry group $G$ of the DMS $(Q,L_d)$, the \emph{discrete momentum map} $J_{d}:Q\times Q\lra \mathfrak{g}^{\ast
} $ is defined by $\displaystyle{
J_{d}\left( q_{0},q_{1}\right) \left( \xi \right) :=-D_{1}L_{d}\left( q_{0},q_{1}\right) \xi _{Q}\left( q_{0}\right)}$,
where $\xi \in \mathfrak{g}:= \Lie G$ and
\[
\xi _{Q}\left( q\right)
:= \dfrac{d}{dt}l_{\exp \left( t\xi \right) }^{Q}\left( q\right) \bigg|_{t=0}
\]
is the infinitesimal generator of the action corresponding to $\xi \in \mathfrak{g}$.
\end{definition}

The next result is the discrete version of Noether's Theorem for discrete mechanical systems (see \cite[Theorem 1.3.3]{ar:marsden_west:2001:discrete_mechanics_and_variational_integrators}).

\begin{theorem}[Discrete Noether's Theorem]
Let $G$ be a symmetry group of the regular DMS $(Q,L_{d})$. Then, the corresponding discrete momentum map $J_{d}:Q\times Q\lra \mathfrak{g}^{\ast }$ is preserved by the flow\footnote{As we mentioned preoviously, the flow of a DMS is usually not globally defined, if it even exists. However, we will always assume that its domain is the whole $Q \times Q$ to avoid making the notation more cumbersome.} of the system $\mathbf{F}_{L_{d}}:Q\times Q\lra Q\times Q$; that is $\displaystyle{J_{d}\circ \mathbf{F}_{L_{d}}=J_{d}}$.
\end{theorem}

\subsection{Reduction of symmetric discrete mechanical systems}

As in the case of continuous Lagrangian systems, it is natural to wonder how to eliminate the symmetries present in a discrete mechanical system. The goal is to obtain a system (which we will call ``reduced'') whose equations of motion may be easier to solve and that, in turn, allow us to reconstruct the original solution.


In order to establish a reduced variational principle and write down reduced equations of motion, it is convenient to work with a space diffeomorphic to the quotient manifold $(Q \times Q)/G$ and a smooth function defined there that replaces the reduced discrete Lagrangian $l_d$. We will see that this can be done in a way that mirrors the continuos case, using an affine discrete connection on the principal bundle $\pi : Q \lra Q/G$.


\subsubsection{Discrete connections}

Discrete connections were introduced in \cite{un:leok_marsden_weinstein:2005:a_discrete_theory_of_connections_on_principal_bundles} and then redefined in \cite{ar:fernandez_tori_zuccalli:2010:lagrangian_reduction_of_discrete_mechanical_systems,ar:fernandez_zuccalli:2013:a_geometric_approach_to_discrete_connections_on_principal_bundles}, where a generalization of these objects (the so--called affine discrete connections) was presented. As in the case of principal connections, these objects can be defined in several ways. We now review some of them and recall their fundamental properties.

Let $\pi : Q \lra Q/G$ be a smooth principal $G$-bundle and $\gamma : Q \lra G$ be a smooth function satisfying $\gamma(l_g^Q(q)) = l_g^G(\gamma(q))\footnote{Recall that $G$ acts on itself by conjugation}$ that we call a \emph{level function}. We also consider the smooth map $\widetilde{\gamma} : Q \lra Q$ defined by $\widetilde{\gamma}(q) := l^Q_{\gamma(q)}(q)$ and its graph $\Gamma := \{ (q,\widetilde{\gamma}(q)) \in Q \times Q \mid q \in Q \}$; naturally, $\Gamma \subs Q \times Q$ is an embedded submanifold. Notice that $\widetilde{\gamma} \circ l_g^Q = l_g^Q \circ \widetilde{\gamma}$ and, also, $\pi \circ \widetilde{\gamma} = \pi$.

\begin{definition}\label{def:affine_discrete_connection}
Let $\Hor \subs Q \times Q$ be an embedded submanifold containing $\Gamma$ that is $G$-invariant for the diagonal action $l^{Q \times Q}$. We say that $\Hor$ is an \emph{affine discrete connection} on $\pi$ with level $\gamma$ if $(\id_Q \times \pi)|_{\Hor} : \Hor \lra Q \times Q/G$ is an injective local diffeomorphism.
\end{definition}

\begin{remark}
In the special case where $\widetilde{\gamma} = \id_Q$, affine discrete connections are called simply \emph{discrete connections} and were studied extensively in \cite{ar:fernandez_zuccalli:2013:a_geometric_approach_to_discrete_connections_on_principal_bundles,ar:fernandez_juchani_zuccalli:2022:discrete_connections_on_principal_bundles_the_discrete_atiyah_sequence,ar:fernandez_kordon:2025:the_integration_problem_for_principal_connections}.
\end{remark}

\begin{remark}\label{rmk:differences_in_definitions_discrete_connections}
Affine discrete connections were first considered in Definition 4.7 of \cite{ar:fernandez_tori_zuccalli:2010:lagrangian_reduction_of_discrete_mechanical_systems}. Still, the definition we introduce now is slightly stronger than the one considered in \cite{ar:fernandez_tori_zuccalli:2010:lagrangian_reduction_of_discrete_mechanical_systems}. Indeed, all submanifolds $\Hor$ satisfying Definition \ref{def:affine_discrete_connection} do satisfy Definition 4.7 of \cite{ar:fernandez_tori_zuccalli:2010:lagrangian_reduction_of_discrete_mechanical_systems}, but there are examples where the opposite fails. In fact, Definition \ref{def:affine_discrete_connection} is equivalent to Definition 4.7 of \cite{ar:fernandez_tori_zuccalli:2010:lagrangian_reduction_of_discrete_mechanical_systems} together with the condition that, if $\Hor \cap l_g^{Q \times Q_2}(\Hor) \neq \varnothing$ then $g = e$, where $l_g^{Q \times Q_2}(q_0,q_1) := (q_0,l_g^Q(q_1))$. This condition can be achieved in many cases by shrinking $\Hor$. We chose this strengthened version of the original notion because it is in line with Definition 2.1 of \cite{ar:fernandez_zuccalli:2013:a_geometric_approach_to_discrete_connections_on_principal_bundles}, with the advantage that most of the results of that paper have an analogue version for affine discrete connections.
\end{remark}

The following result is from Proposition 2.4 of \cite{ar:fernandez_zuccalli:2013:a_geometric_approach_to_discrete_connections_on_principal_bundles}.

\begin{proposition}\label{prop:discrete_connections_frakUs}
Let $\Hor$ be an affine discrete connection with level $\gamma$ on $\pi$. Then, the following assertions are true.
\begin{enumerate}
\item The subsets $\frakU := l_G^{Q \times Q_2}(\Hor) \subs Q \times Q$, $\frakU' := (\id_Q \times \pi)(\Hor) \subs Q \times (Q/G)$ and $\frakU'' := (\pi \times \pi)(\Hor) \subs (Q/G) \times (Q/G)$.

\item $\frakU$ is $G \times G$-invariant and $\frakU = (\pi \times \pi)^{-1}(\frakU'')$.
\end{enumerate}
\end{proposition}

Just as principal connections can be characterized by their connection $1$-form, we give next a characterization of affine discrete connections in the same spirit.

\begin{lemma}\label{lemma:unique_g}
Let $\Hor \subs Q \times Q$ define an affine discrete connection on $\pi$ with level function $\gamma$. For any $(q_0,q_1) \in \frakU$ (the open subset defined in Proposition \ref{prop:discrete_connections_frakUs}), there is a unique $g \in G$ such that $l^{Q \times Q_2}_{g^{-1}}(q_0,q_1) \in \Hor$.
\end{lemma}
\begin{proof}
The existence of $g \in G$ is due to the definition of $\frakU$. Its uniqueness follows from the injectivity of $(\id_Q \times \pi)|_{\Hor}$.
\end{proof}

\begin{definition}
The map $\calA_d : \frakU \lra G$ defined by $\calA_d(q_0,q_1) := g$, where $g$ is the element of $G$ associated to $(q_0,q_1)$ by Lemma \ref{lemma:unique_g}, is called the \emph{affine discrete connection form} associated to the affine discrete connection defined by $\Hor$.
\end{definition}

We can also give a direct formula for $\calA_d$,
\[
\calA_d(q_0,q_1) = \kappa(\pr_2(((\id_Q \times \pi)|_{\Hor})^{-1} ((\id_Q \times \pi)(q_0,q_1))), q_1),
\]
where $\kappa : Q \times_G Q \lra G$ is the smooth map with the property that $\kappa(q,q') = g$ if and only if $l_g^Q(q) = q'$; recall that $Q\times_G Q$ is the fiber product, whose elements are the pairs $(q,q')\in Q\times Q$ such that $\pi(q)=\pi(q')$. This direct formula shows that $\calA_d$ is a smooth map. The following is an important characterization of affine discrete connections.

\begin{theorem}\label{thm:calAd_equivariance}
Let $\Hor_{\calA_d}$ be an affine discrete connection on $\pi$ with level function $\gamma$ and domain $\frakU$. Then, for all $(q_0,q_1) \in \frakU$ and $g_0,g_1 \in G$,
\begin{equation}\label{eq:calAd_equivariance}
\calA_d(l_g^Q(q_0),l_g^Q(q_1)) = g_1 \calA_d(q_0,q_1) g_0^{-1}.
\end{equation}

In addition, $\Hor_{\calA_d} = \{ (q_0,q_1) \in \frakU \mid \calA_d(q_0,q_1) = e \}$. Conversely, given a smooth function $\calA : \mathcal{U} \lra G$, where $\mathcal{U} \subs Q \times Q$ is an open subset containing $\Gamma$ and is $G \times G$-invariant, such that \eqref{eq:calAd_equivariance} (with $\calA_d$ replaced by $\calA$) and $\calA(\Gamma) = \{ e \}$, then $\Hor := \{ (q_0,q_1) \in \mathcal{U} \mid \calA(q_0,q_1) = e \}$ defines an affine discrete connection with level function $\gamma$ whose associated discrete connection form is $\calA$.
\end{theorem}
\begin{proof}
This is, almost verbatim, the proof of Theorem 3.4 of \cite{ar:fernandez_zuccalli:2013:a_geometric_approach_to_discrete_connections_on_principal_bundles}.
\end{proof}

Another useful characterization of connections in a principal bundle is given by the horizontal lift. Next, we introduce a similar characterization for affine discrete connections.

\begin{definition}
Let $\Hor_{\calA_d}$ be an affine discrete connection with domain $\frakU$ on $\pi$ with level function $\gamma$; let $\frakU' \subs Q \times (Q/G)$ be the open subset defined in Proposition \ref{prop:discrete_connections_frakUs}. The \emph{associated affine discrete horizontal lift} $h_d : \frakU' \lra Q \times Q$ is the inverse map of the diffeomorphism $(\id_Q \times \pi)|_{\Hor_{\calA_d}}^{\frakU'} : \Hor_{\calA_d} \lra \frakU'$. Explicitly,
\[
h_d(q_0,\tau_1) = (q_0,q_1) \quad \Longleftrightarrow \quad (q_0,q_1) \in \Hor_{\calA_d} \text{ and } \pi(q_1) = \tau_1.
\]

It is convenient to write $h_d^{q_0}(\tau_1) := h_d(q_0,\tau_1)$ and define $\overline{h_d^{q_0}} := \pr_2 \circ h_d^{q_0}$.
\end{definition}

\begin{remark}
When $(q_0,q_1) \in \frakU$, if $g := \calA_d(q_0,q_1)$, we know that $(q_0,l^Q_{g^{-1}}(q_1)) \in \Hor_{\calA_d}$. Then, $h_d^{q_0}(\pi(q_1)) = (q_0, l^Q_{g^{-1}}(q_1))$ and $\overline{h_d^{q_0}}(\pi(q_1)) = l^Q_{g^{-1}}(q_1)$, so that
\[
l^Q_g(\overline{h_d^{q_0}}(\pi(q_1))) = q_1.
\]

Thus,
\begin{equation}
\calA_d(q_0,q_1) = \kappa (\overline{h_d^{q_0}}(\pi(q_1)),q_1) \quad \text{for all } (q_0,q_1) \in \frakU.
\end{equation}
\end{remark}

The following result establishes the basic properties of the discrete horizontal lift associated to an affine discrete connection. It also proves that an affine discrete connection can be reconstructed out of its associated affine discrete horizontal lift.

\begin{theorem}\label{thm:horizontal_lift_properties}
Let $\calA_d$ be an affine discrete connection on $\pi : Q \lra Q/G$ with domain $\frakU$ and level function $\gamma$. Then, the following assertions are true.
\begin{enumerate}
\item $\frakU' \subs Q \times (Q/G)$ is $G$-invariant for the action defined by $l^{Q \times (Q/G)}_g (q,\tau) := (l^Q_g(q),\tau)$.

\item $h_d : \frakU' \lra Q \times Q$ is smooth and $G$-equivariant for the $G$-actions $l^{Q \times (Q/G)}$ and $l^{Q \times Q}$.

\item $h_d$ is a section over $\frakU'$ of $\id_Q \times \pi : Q \times Q \lra Q \times (Q/G)$, that is, $(\id_Q \times \pi) \circ h_d = \id_{\frakU'}$.

\item For every $q \in Q$, $(q,\pi(q)) \in \frakU'$ and $h_d^q(\pi(q)) = (q,\widetilde{\gamma}(q))$.
\end{enumerate}

Conversely, assume that $\mathcal{U}' \subs Q \times (Q/G)$ is an open set that satisfies condition ({\it 1}) (with $\frakU'$ replaced by $\mathcal{U}'$) and $h : \mathcal{U}' \lra Q \times Q$ is a map such that conditions ({\it 2}), ({\it 3}) and ({\it 4}) are satisfied (with $\frakU'$ and $h_d$ replaced by $\mathcal{U}'$ and $h$). Then, there exists a unique affine discrete connection $\calA_d$ with domain $\frakU := (\id_Q \times \pi)^{-1}(\mathcal{U}')$ on $\pi$ and level function $\gamma$ such that $\mathcal{U}' = \frakU'$ and $h = h_d$.
\end{theorem}
\begin{proof}
This is, almost verbatim, the proof of Theorem 4.4 in \cite{ar:fernandez_zuccalli:2013:a_geometric_approach_to_discrete_connections_on_principal_bundles}.
\end{proof}

\subsubsection{Reduced spaces and reduced Lagrangians}\label{subsubsec:reduced_spaces_and_lagrangians_general}

We will now use an affine discrete connection on the principal bundle $\pi$ to construct a smooth manifold diffeomorphic to the quotient manifold $(Q \times Q)/G$. To keep the notation simple, in what follows we will always asume that the affine discrete connections are defined on the whole space.

\begin{definition}
Consider a principal bundle $\pi$ and the action of a Lie group $G$ on $Q \times G$ given by $l_g^{Q \times G}(q,w) := (l_g^Q(q),l_g^G(w))$. Since $\pi$ is a principal bundle, the quotient $\tilde{G} := (Q \times G)/G$ is a smooth manifold and the quotient map denoted by $\rho : Q \times G \lra \tilde{G}$ is smooth.

In addition, $\tilde{G}$ is a fiber bundle over $Q/G$ called \emph{associated conjugante bundle}, whose projection induced by $\pr_1$ will be denoted by $\pr_{Q/G}$.
\end{definition}

An affine discrete connection $\mathcal{A}_d$ allows us to identify the space $(Q\times Q)/G$ with the product manifold $\tilde{G} \times (Q/G)$, as can be seen in \cite[Proposition 4.19]{ar:fernandez_tori_zuccalli:2010:lagrangian_reduction_of_discrete_mechanical_systems}.

\begin{proposition}\label{affine-isom}
Given an affine discrete connection $\mathcal{A}_d$ on $\pi:Q\lra Q/G$, let $\tilde{{\Phi}}_{\mathcal{A}_d}:Q\times Q \lra Q\times G \times (Q/G)$ and $\tilde{{\Psi}}_{\mathcal{A}_d} : Q\times G \times (Q/G)\lra Q\times Q$ be the maps defined by
\[
\tilde{{\Phi}}_{\mathcal{A}_d}(q_0,q_1) := (q_0,\mathcal{A}_d(q_0,q_1), \pi(q_1)) \ \ \mbox{and}\  \ \tilde{{\Psi}}_{\mathcal{A}_d}(q_0,w_0,\tau_1) := (q_0,\tilde{F}_1(q_0,w_0,\tau_1)),
\]
where $\tilde{F}_1 : Q \times G \times (Q/G) \lra Q$ is the map given by $(q,w,\tau) \mapsto l_{w}^Q \left( \overline{h_d^{q}} (\tau) \right)$.

Then, $\tilde{{\Phi}}_{\mathcal{A}_d}$ and $\tilde{{\Psi}}_{\mathcal{A}_d}$ are smooth and mutually inverses. In addition, considering the diagonal action of $G$ on $Q\times Q$ and
\[
l^{Q\times G \times (Q/G)}_g(q_0,w_0,\tau_1):=(l^Q_g(q_0),l^G_g(w_0),\tau_1),
\]
both maps are $G$-equivariant and, hence, induce the diffeomorphisms $\Phi_{\mathcal{A}_d}:(Q\times Q)/G \lra \tilde{G} \times (Q/G) $ and $\Psi_{\mathcal{A}_d} : \tilde{G} \times (Q/G) \lra (Q\times Q)/G$.
\end{proposition}

The proof of Proposition 4.19 given in \cite{ar:fernandez_tori_zuccalli:2010:lagrangian_reduction_of_discrete_mechanical_systems} relies on the notion of affine discrete connection introduced in that paper. Still, it only uses the properties stated in Theorems \ref{thm:calAd_equivariance} and \ref{thm:horizontal_lift_properties} which, we saw, remain valid for affine discrete connections in the sense of Definition \ref{def:affine_discrete_connection}.

Using the diffeomorphism $\tilde{\Phi}_{\calA_d}$, we can define the function $\check{L}_d : Q \times G \times (Q/G) \lra \rr$ by $\check{L}_d := L_d \circ (\tilde{\Phi}_{\calA_d})^{-1}$. Since $\check{L}_d$ is $G$-invariant, it induces a function that we may call product reduced Lagrangian $\hat{L}_d$ on $\tilde{G} \times (Q/G)$ given by
$$\hat{L}_d(\rho(q_0,w_0),\tau_1) := \check{L}_d(q_0,w_0,\tau_1)=L_d(q_0,\tilde{F}_1(q_0,w_0,\tau_1)).$$

Considering the map $\Upsilon:Q \times Q \lra \tilde{G} \times (Q/G)$ defined by
$$
\Upsilon(q_0,q_1):=(\rho(q_0,\mathcal{A}_d(q_0,q_1)), \pi(q_1)),
$$
the following diagram commutes:
\begin{equation*}
\xymatrixcolsep{5pc}\xymatrixrowsep{4pc}\xymatrix{
& & \rr \\ 
Q \times Q \ar@/^/[urr]^{L_d} \ar[d]_{\tilde{\pi}}
\ar[r]^{\tilde{\Phi}_{\mathcal{A}_d}} \ar[dr]_{\Upsilon} & Q \times G \times (Q/G) \ar[ur]^(.4){\check{L}_d} \ar[d]^{\rho \times \id_{Q/G}} & \\
(Q\times Q)/G \ar[r]_{\Phi_{\mathcal{A}_d}} & \tilde{G} \times (Q/G) \ar@/_1.5pc/[uur]_{\hat{L}_d} &
}
\end{equation*}

\subsubsection{Reduced dynamics}

We now relate the dynamics of the DMS $(Q,L_d)$ with the dynamics of the reduced system on $\tilde{G}\times Q/G$.
We are going to establish a reduced variational principle considering the reduced Lagrangian $\hat{L}_d:\tilde{G} \times (Q/G)\lra \rr$, the discrete action
\[
\hat{S}_d(v.,\tau.):=\sum_{k=0}^{N-1} \hat{L}_d(v_k,\tau_{k+1}),
\]
where $(v.,\tau.) := ((v_0,\tau_1),\ldots,(v_{N-1},\tau_N))$, and an affine discrete connection $\mathcal{A}_{d}$ on the principal bundle $\pi : Q \lra Q/G$. To obtain equations of motion for the reduced system, the idea will be to use that when we consider the product manifold $Q \times G \times (Q/G)$, the differential of $\check{L}_d$ is decomposed as a sum of differentials (one for each factor), as we previously mentioned (see Section \ref{subsection:preliminaries_on_product_manifolds}).

We can relate the discrete actions $\hat{S}_d$ and $S_d$ as follows. Given a discrete curve $q. := (q_0,\ldots,q_N)$ on $Q$, let $(v.,\tau.)$ be the discrete curve on $\tilde{G} \times (Q/G)$ defined by $\tau_{k}:=\pi \left( q_{k}\right) $, $w_{k}:=\mathcal{A}_{d}\left( q_{k},q_{k+1}\right) $ and $v_{k}:=\rho \left( q_{k},w_{k}\right) $ (that is, $(v_k,\tau_k) := \Upsilon(q_k,q_{k+1})$). If $\delta q.$ is an infinitesimal variaton of the curve $q.$, we define the infinitesimal variation $(\delta v.,\delta \tau.)$ of $(v.,\tau.)$ by $\delta \tau_k := T_{q_k}\pi(\delta q_k)$, $1 \le k \le N$, and $\delta v_k := T_{(q_k,w_k)}\rho(\delta q_k,\delta w_k)$, $0 \le k \le N-1$, with $\delta w_k := T_{(q_k,q_{k+1})}\calA_d\left( \delta q_{k},\delta q_{k+1}\right)$.

It is easy to see that for these curves and infinitesimal variations,
\[
dS_d(q.)(\delta q.)=d\hat{S}_{d}\left( v_{\cdot },\tau_{\cdot}\right) \left( \delta v_{\cdot },\delta \tau_{\cdot }\right).
\]

Indeed, a simple computation shows that
\[
\begin{split}
T_{\left( q_{k},q_{k+1}\right)}&\Upsilon \left(
\delta q_{k},\delta q_{k+1}\right) = \\
& \left( T_{(q_k,\calA_d(q_k,q_{k+1}))}\rho \left(
\delta q_{k},T_{(q_k,q_{k+1})}\calA_d \left( \delta
q_{k},\delta q_{k+1}\right),T_{q_{k+1}}\pi(\delta q_{k+1})\right) \right).
\end{split}
\]

Hence,
\begin{equation*}\label{forcedVP}
\begin{split}
dS_d(q.)(\delta q.) &=  \sum_{k=0}^{N-1} dL_{d}\left( q_{k},q_{k+1}\right) \left( \delta
q_{k},\delta q_{k+1}\right) \\
&=  \sum_{k=0}^{N-1}d\hat{L}_{d}\left( \Upsilon \left(
q_{k},q_{k+1}\right) \right) \left( T_{\left( q_{k},q_{k+1}\right)}\Upsilon \left(
\delta q_{k},\delta q_{k+1}\right) \right) \\
& = \sum_{k=0}^{N-1}d\hat{L}_{d}\left( v_{k},\tau_{k+1}\right) \left( T_{(q_k,\calA_d(q_k,q_{k+1}))}\rho \left(
\delta q_{k}, T_{(q_k,q_{k+1})}\calA_d \left( \delta
q_{k},\delta q_{k+1} \right) \right), \right. \\
& \ \ \ \ \ \ \ \ \left. T_{q_{k+1}}\pi (\delta q_{k+1}) \right) \\
& = \sum_{k=0}^{N-1}d\hat{L}_{d}\left( v_{k},\tau_{k+1}\right) \left( \delta v_{k},\delta \tau_{k+1} \right) = d\hat{S}_{d}\left( v_{\cdot },\tau_{\cdot}\right) \left( \delta v_{\cdot },\delta \tau_{\cdot }\right).
\end{split}
\end{equation*}

In other words, we have the following result:

\begin{lemma}\label{lemma:q trajectory iff vtau satisfies red var principle}
Let $(Q,L_d)$ be a DMS. Using the previous notation for the discrete curves and infinitesimal variations,
\[
dS_d(q.)(\delta q.)=d\hat{S}_{d}\left( v_{\cdot },\tau_{\cdot}\right) \left( \delta v_{\cdot },\delta \tau_{\cdot }\right).
\]
\end{lemma}

So far, the infinitesimal variations $(\delta v., \delta \tau.)$ are defined in terms of infinitesimal variations $\delta q_\cdot$. The next step will be to give a description of these ``reduced variations" in terms that are intrinsic to the reduced space.

In order to do this, in addition to the previous elements, we will consider a principal connection $\frakA$ on the principal bundle $\pi$ with horizontal subbundle $\Hor_\frakA$. Hence, $TQ = \mathcal{V}^\SG \oplus \Hor_\frakA$, where the \emph{vertical subbundle} $\mathcal{V}^\SG$ over $Q$ is a subbundle of $TQ$ with fibers $\mathcal{V}^\SG_q:=T_q(l^Q_\SG(\{q\})) = \ker (T_q\pi)$.
The associated \emph{horizontal lift} is the map $h : Q \times_{Q/G} T(Q/G) \lra TQ$ defined by
$\displaystyle{h^q(w_{\pi(q)}) = v_q \ \text{if} \ v_q \in \Hor_\frakA(q) \ \text{and} \ T_q\pi ( v_q )= w_{\pi(q)}}$, where $\displaystyle{Q \times_{Q/G} T(Q/G)}$ is the fibered product over $Q/G$ of the fiber bundles $\pi$ and the tangent bundle of $Q/G$.

This connection will allow us to split the infinitesimal variations $\delta v_k$ in a convenient way to express the equations of motion of the reduced system, as can be seen in the proof of Theorem 5.11 of \cite{ar:fernandez_tori_zuccalli:2010:lagrangian_reduction_of_discrete_mechanical_systems}. For the moment, we use it to establish the equivalence between the reduced variational principle and the discrete Hamilton principle.

\begin{proposition}\label{prop:q trajectory iff vtau satisfies red var principle}
Given a DMS $(Q,L_d)$ with symmetry group $\SG$, let $q_\cdot := (q_0,\ldots,q_N)$ be a discrete curve on $Q$ and $\tau_k:=\pi(q_k)$, $w_k:=\CD(q_k,q_{k+1})$ and $v_k:=\rho(q_k,w_k)$ be the corresponding discrete curves on $Q/\SG$, $\SG$ and $\ti{\SG}$, respectively. Then, the following statements are equivalent:
\begin{enumerate}
\item $q.$ satisfies the variational principle $dS_d(q_\cdot)(\delta q_\cdot) = 0$ for all infinitesimal variations with fixed endpoints $\delta q_\cdot$ of $q_\cdot$.

\item $(v.,\tau_\cdot)$ satisfies the variational principle
\[
d\hat{S}_d(v.,\tau_\cdot) (\delta v., \delta \tau_\cdot) = 0
\]
for every infinitesimal variation $(\delta v_\cdot, \delta \tau_\cdot)$ of $(v., \tau_\cdot)$ satisfying:
\begin{enumerate}
	\item $\delta \tau_k \in T_{\tau_k}(Q/G)$ for all $k = 1,\ldots,N$,
	\item for $k = 0,\ldots,N-1$,
	\begin{equation}\label{eq:delta_vk-def}
		\begin{split}
			\delta v_k :=& \; T_{(q_k,w_k)}\rho \big(\HLc{q_k}(\delta \tau_k),
			T_{(q_k,q_{k+1})}\CD(\HLc{q_k}(\delta \tau_k),\HLc{q_{k+1}}(\delta
			\tau_{k+1}))\big) \\
			& + T_{(q_k,w_k)}\rho\big((\xi_k)_Q(q_k),
			T_{(q_k,q_{k+1})}\CD((\xi_k)_Q(q_k),(\xi_{k+1})_Q(q_{k+1}))\big),
		\end{split}
	\end{equation}
	where $\xi. := (\xi_0,\ldots,\xi_N) \in \frakg^{N+1}$ and $h$ is the horizontal lift associated to $\frakA$, $\delta \tau_0 \in T_{\pi(q_0)}(Q/G)$,
	\item and the fixed endpoint conditions: $\delta \tau_0 = 0$\hspace{.1pc}\footnote{Notice that $\delta {\tau}_{0}$ is just a convenient device to simplify the expression of the $\delta v_k$.}, $\delta \tau_N = 0$, $\xi_0 = \xi_N = 0$.
\end{enumerate}
\end{enumerate}
\end{proposition}
\begin{proof}
\leavevmode

{\it 2} $\Lra$ {\it 1}. Given an infinitesimal variation $\delta q.$ as in the statement, define $\delta \tau_k := T_{q_k}\pi(\delta q_k)$ and $\delta v_k$ using \eqref{eq:delta_vk-def} for $\xi_k := \frakA(q_k)(\delta q_k)$, so that $\delta q_k = h^{q_k}(\delta \tau_k) + (\xi_k)_Q(q_k)$. Noticing that $(\delta v., \delta \tau.)$ satisfies the conditions of item {\it (2)}, Lemma \ref{lemma:q trajectory iff vtau satisfies red var principle} implies that $dS_d(q.)(\delta q.) = 0$.

{\it 1} $\Lra$ {\it 2}. Given $(\delta v., \delta \tau.)$ as in the statement, let $\delta q_k := h^{q_k}(\delta \tau_k) + (\xi_k)_Q(q_k)$. Since the infinitesimal variation $\delta q.$ has fixed endpoints, Lemma \ref{lemma:q trajectory iff vtau satisfies red var principle} implies that $d\hat{S}_d(v.,\tau.)(\delta v., \delta \tau.) = 0$.
\end{proof}

Finally, we establish a reduction process of a $G$-symmetry of the DMS $(Q,L_d)$ that turns out to be a discrete version of the Lagrangian reduction theorem presented in \cite{ar:cendra_marsden_ratiu:2001:geometric_mechanics_lagrangian_reduction_and_nonholonomic_systems}.

\begin{remark}\label{remark:notation_lie_groups}
The analysis of symmetric systems leads us to consider objects defined on a Lie group $G$. The statement of the following theorem requires a convenient notation used in this context, which we briefly review now. Given $w_0,w_1 \in G$ and $\delta w_0 \in T_{w_0}G$, we define
\[
w_1 \delta w_0 := T_{w_0}L_{w_1} (\delta w_0) \in T_{w_1 w_0}G, \quad \delta w_0 w_1 := T_{w_0}R_{w_1}(\delta w_0) \in T_{w_0 w_1}G,
\]
where $L_{w_1}$ and $R_{w_1}$ denote the left and right translation by $w_1$, respectively. Analogously, when $\alpha_0 \in T_{w_0}^*G$,
\[
w_1 \alpha_0 := (L_{w_1})^*(\alpha_0) \in T_{w_1 w_0}^*G, \quad \alpha_0 w_1 := (R_{w_1})^*(\alpha_0) \in T_{w_0 w_1}^*G.
\]
\end{remark}

\begin{theorem}\label{th:4_points-general} Given a DMS $(Q,L_d)$ with symmetry group $\SG$, let $q_\cdot := (q_0,\ldots,q_N)$ be a discrete curve on $Q$ and $\tau_k:=\pi(q_k)$, $w_k:=\CD(q_k,q_{k+1})$ and $v_k:=\rho(q_k,w_k)$ be the corresponding discrete curves on $Q/\SG$, $\SG$ and $\ti{\SG}$, respectively. Then, the following statements are equivalent:
\begin{enumerate}
\item \label{it:var_pple-general} $q_\cdot$ satisfies the variational principle $dS_d(q_\cdot)(\delta q_\cdot) = 0$ for all infinitesimal variations with fixed endpoints $\delta q_\cdot$ of $q_\cdot$.

\item \label{it:eq_lda-general} $q_\cdot$ satisfies the discrete Euler--Lagrange equations \eqref{dELe}.

\item \label{it:red_var_pple-general} $(v.,\tau_\cdot)$ satisfies the variational principle
\[
d\hat{S}_d(v.,\tau_\cdot) (\delta v., \delta \tau_\cdot) = 0
\]
for every infinitesimal variation $(\delta v_\cdot, \delta \tau_\cdot)$ and $\xi. := (\xi_0,\ldots,\xi_N) \in \frakg^{N+1}$ satisfying:
\begin{enumerate}
	\item $\delta \tau_k \in T_{\tau_k}(Q/G)$ for all $k = 1,\ldots,N$,
	\item for $k = 0,\ldots,N-1$,
\begin{equation*}
\begin{split}
\delta v_k :=& \; T_{(q_k,w_k)}\rho\big(\HLc{q_k}(\delta \tau_k),
T_{(q_k,q_{k+1})}\CD(\HLc{q_k}(\delta \tau_k),\HLc{q_{k+1}}(\delta
\tau_{k+1}))\big) \\
& + T_{(q_k,w_k)}\rho\big((\xi_k)_Q(q_k),
T_{(q_k,q_{k+1})}\CD((\xi_k)_Q(q_k),(\xi_{k+1})_Q(q_{k+1}))\big),
\end{split}
\end{equation*}
where $h$ is the horizontal lift associated to $\frakA$, $\delta \tau_0 \in T_{\pi(q_0)}(Q/G)$,
	\item and the fixed endpoint conditions: $\delta \tau_0 = 0$\hspace{.1pc}\footnote{Notice that $\delta {\tau}_{0}$ is just a convenient device to simplify the expression of the $\delta v_k$.}, $\delta \tau_N = 0$, $\xi_0 = \xi_N = 0$.
\end{enumerate}

\item \label{it:red_lda_eq-general}
$(v.,\tau_\cdot)$ satisfies the following conditions for each $(v_{k-1},\tau_k,v_k,\tau_{k+1})$:
\begin{itemize}
\item $\phi = 0$ for $\phi \in T_{\tau_k}^*(Q/G)$ defined by
\[
\begin{split}
\phi &:= D_1 \check{L}_d(q_k,w_k,\tau_{k+1}) \circ \HLc{q_k} + D_3 \check{L}_d(q_{k-1},w_{k-1},\tau_{k}) \\
& \ \ \ + D_2 \check{L}(q_k,w_k,\tau_{k+1}) D_1\calA_d(q_k,q_{k+1}) \circ h^{q_k} \\
& \ \ \ + D_2 \check{L}(q_{k-1},w_{k-1},\tau_k) D_2\calA_d(q_{k-1},q_k) \circ h^{q_k}.
\end{split}
\]

\item $\psi(\xi_k) = 0$ for all $\xi_k \in \frakg$, where $\psi \in \frakg^*$ is defined by
\[
\psi := D_2\check{L}_d(q_{k-1},w_{k-1},\tau_{k}) w_{k-1}^{-1} - D_2\check{L}_d(q_k,w_k,\tau_{k+1}) w_k^{-1},
\]
where we are using the notation described in Remark \ref{remark:notation_lie_groups}.
\end{itemize}
\end{enumerate}
\end{theorem}
\begin{proof}
\leavevmode

{\it 1} $\Longleftrightarrow$ {\it 2}. It is a straightforward result that can be found, for example, in \cite[Part I]{ar:marsden_west:2001:discrete_mechanics_and_variational_integrators}.

{\it 3} $\Longleftrightarrow$ {\it 1}. It is Proposition \ref{prop:q trajectory iff vtau satisfies red var principle}.

{\it 3} $\Longleftrightarrow$ {\it 4}. For infinitesimal variations
  $(\delta v_\cdot,\delta \tau_\cdot)$, writing
  $\delta v_k = T_{(q_k,w_k)} \rho (\delta q_k , \delta w_k)$ and
  following the proof of Theorem 5.11 of \cite{ar:fernandez_tori_zuccalli:2010:lagrangian_reduction_of_discrete_mechanical_systems}, we obtain
  \begin{equation}\label{dShat}
    \begin{split}
      d\hat{S}_d(v_\cdot,\tau_\cdot)&(\delta v_\cdot,\delta \tau_\cdot) = \sum_{k=1}^{N-1} \left( D_1 \check{L}_d(q_k,w_k,\tau_{k+1}) \circ h^{q_k} + D_3 \check{L}_d(q_{k-1},w_{k-1},\tau_k) \right. \\
      & + D_2 \check{L}_d(q_k,w_k,\tau_{k+1}) D_1\calA_d(q_k,q_{k+1}) \circ h^{q_k} \\
      & + \left. D_2 \check{L}_d(q_{k-1},w_{k-1},\tau_k) D_2\calA_d(q_{k-1},q_k) \circ h^{q_k} \right) (\delta \tau_k) \\
      & + \sum_{k=0}^{N-1} \left( D_1\check{L}_d(q_k,w_k,\tau_{k+1}) \left( (\xi_k)_Q(q_k) \right) \right. \\
      & + \left. D_2\check{L}_d(q_k,w_k,\tau_{k+1}) T_{(q_k,q_{k+1})}
        \calA_d \left( (\xi_k)_Q(q_k) , (\xi_{k+1})_Q(q_{k+1}) \right)
      \right).
    \end{split}
  \end{equation}
	
Since the variations $\delta \tau_\cdot$ are independent of those generated by $\xi.$, we have that point $3$ in the statement is equivalent to the vanishing of the summation involving $\delta \tau_\cdot$ and the one involving $\xi.$ independently. The first of these conditions, since the $\delta \tau_\cdot$ are free, is equivalent to $\phi$ being identically zero.

For the second summation, we first perform the following computations:
\[
\begin{split}
T_{(q_k,q_{k+1})}\calA_d &\left( (\xi_k)_Q(q_k) , (\xi_{k+1})_Q(q_{k+1}) \right) \\
&= \frac{d}{dt} \bigg|_{t=0} \left( \exp(t \xi_{k+1}) \calA_{d}(q_k,q_{k+1}) \exp(-t \xi_k) \right) \\
&= \xi_{k+1} w_k - w_k \xi_k,
\end{split}
\]
where we are using the notation described in Remark \ref{remark:notation_lie_groups}.

On the other hand, since $\check{L}_d(l_g^Q(q_k),w_k,\tau_{k+1}) = \check{L}_d(q_k,l_{g^{-1}}^G(w_k),\tau_{k+1})$,
\[
\begin{split}
D_1\check{L}_d(q_k,w_k,\tau_{k+1}) \left( (\xi_k)_Q(q_k) \right) &= \frac{d}{dt} \bigg|_{t=0} \check{L}_d(l_{\exp(t \xi_k)}^Q(q_k),w_k,\tau_{k+1}) \\
&= D_2 \check{L}_d(q_k,w_k,\tau_{k+1}) \frac{d}{dt} \bigg|_{t=0} (\exp(-t \xi_k) w_k \exp(t\xi_k)) \\
&= D_2 \check{L}_d(q_k,w_k,\tau_{k+1}) (-\xi_k w_k + w_k \xi_k).
\end{split}
\]

Using both computations, we conclude that the vanishing of the second summation in \eqref{dShat} is equivalent to
\[
\begin{split}
0 &= \sum_{k=1}^{N-1} D_2 \check{L}_d(q_k,w_k,\tau_{k+1}) (\xi_{k+1} w_k - \xi_k w_k) \\
&= \sum_{k=1}^{N-1} \left( D_2 \check{L}_d(q_{k-1},w_{k-1},\tau_k) w_{k-1}^{-1} - D_2 \check{L}_d(q_k,w_k,\tau_{k+1}) w_k^{-1} \right) (\xi_k).
\end{split}
\]
\end{proof}

\begin{remark}
The previous result was presented in \cite{ar:caruso_fernandez_tori_zuccalli:2026:remarks_on_structures_and_preservation_in_forced_discrete_mechanical_systems_of_Routh_type} as a particular case of the reduction process for forced discrete mechanical systems considered in \cite[Theorem 3.4]{ar:caruso_fernandez_tori_zuccalli:2023:lagrangian_reduction_of_forced_discrete_mechanical_systems}. As we will see later, this process can also be considered as a particular case of the reduction process for discrete mechanical systems with nonholonomic constraints developed in \cite{ar:fernandez_tori_zuccalli:2010:lagrangian_reduction_of_discrete_mechanical_systems}, considering the kinematic constraints $\mathcal{D}_{d}=Q \times Q$ and the variational constraints $\mathcal{D}=TQ$, as well as a special case of the reduction by stages process discussed in Section \ref{section:reduction by stages}.
\end{remark}

\begin{remark}
The reduced system we obtained in Theorem \ref{th:4_points-general} usually fails to be a DMS, since the quotient manifold $(Q \times Q)/G$ and its diffeomorphic model $\tilde{G} \times Q/G$ are not necessarily a product manifold of the form $M \times M$. An immediate consequence of this fact is that we cannot perform a second reduction process if there is a residual symmetry. This limitation is overcome constructing a suitable category whose objects are general enough to include both discrete mechanical systems and the dynamical systems obtained by the reduction process considered in Theorem \ref{th:4_points-general}. This will be the content of Section \ref{section:reduction by stages}.
\end{remark}

\begin{remark}
Theorem \ref{th:4_points-general} starts from a path $q_\cdot$ and constructs the reduced paths $\tau_\cdot$ and $v_\cdot$. A reverse process can also be considered. Given paths $\tau_\cdot$ and $v_\cdot$ such that $\pr^{Q/G}(v_{k+1}) = \tau_{k+1}$, they can be lifted to a path $q_\cdot$ that is unique once we fix $q_0$. Combining this reconstruction process with Theorem \ref{th:4_points-general}, one has on one hand a discrete curve $q.$ and on the other a curve $(v.,\tau.)$ (regardless of where one started) that satisfy that one is a trajectory if and only if the other one is a trajectory. A detailed discussion can be found in \cite[Section 7]{ar:fernandez_tori_zuccalli:2010:lagrangian_reduction_of_discrete_mechanical_systems}.
\end{remark}



\subsection{Discrete Routh reduction}

In this setting, we will consider the well--known Routh reduction, a reduction process using the preservation of the momentum map.

\subsubsection{Momentum map and connection}

A $G$-symmetry of a DMS $(Q,L_d)$ can be used to define a discrete momentum map that possesses interesting properties. Now, we recall some of them, considered in \cite{ar:marsden_west:2001:discrete_mechanics_and_variational_integrators,ar:fernandez_tori_zuccalli:2010:lagrangian_reduction_of_discrete_mechanical_systems} (see also \cite[Proposition 3.7]{ar:caruso_fernandez_tori_zuccalli:2026:remarks_on_structures_and_preservation_in_forced_discrete_mechanical_systems_of_Routh_type}).

\begin{proposition}\label{momentun map}
Let $G$ be a symmetry group of the DMS $(Q,L_d)$ and let $J_d$ be its associated discrete momentum map. Then,
\begin{enumerate}
\item $J_d$ is $G$-equivariant with respect to the diagonal action and the coadjoint action $Ad^{*}$ on $\frak g^*$. That is, $ J_d(l_g^{Q\times Q}(q_0,q_1))= Ad^{*}_{g^{-1}} (J_d(q_0, q_1))$ for every $g\in G$ and every $(q_0,q_1) \in Q \times Q$.

\item If $L_d$ is regular and $\mu \in \frak g^*$ is a regular value of $J_d$, the subset $J_d^{-1}(\mu)\subs Q\times Q$ is a submanifold of $Q\times Q$.
\end{enumerate}
\end{proposition}

The Discrete Noether's Theorem assures that any trajectory $q.$ of a DMS $(Q,L_d)$ such that $J_d(q_0,q_1)=\mu$ must be contained in the submanifold $J_d^{-1}(\{ \mu \}) \subs Q\times Q$. Notice that $J_d^{-1}(\{ \mu \})$ is not $G$-invariant, but $G_\mu$-invariant, where $G_\mu$ is the isotropy subgroup of $\mu$, given by
\[
G_{\mu} :=\{g\in  G:Ad^*_{g^{-1}}(\mu) = \mu\}\subs G.
\]

It is therefore to be expected that one can consider the action of $G_\mu$ and study its possible reduction. In addition, $G_\mu$ is a symmetry group of $(Q,L_d)$ that acts freely and properly on $Q$ (because $G$ does) and, then, we can consider the smooth manifold $Q/G_{\mu}$ and the principal bundle $\pi_\mu : Q \lra Q/G_\mu$, with structure group $G_{\mu}$. The key idea is to analyze if the submanifold $J_d^{-1}(\{ \mu \})$ can be the horizontal submanifold of an affine discrete connection on $\pi_\mu$.

Let us consider the discrete momentum map associated to this $G_{\mu}$-symmetry, $J_\mu : Q \times Q \lra \frakg_\mu^*$, which is given by $ J_\mu(q_0,q_1)\xi := -D_1 L_d(q_0,q_1)\xi_Q(q_0)$ for all $\xi \in \frakg_\mu$, where $\frakg_\mu := \Lie(G_{\mu})$. We will see that if the $G_{\mu}$-symmetry satisfies certain conditions, then the submanifold $J_\mu^{-1}(\{ \mu \})$ defines an affine discrete connection on the $G_{\mu}$-principal bundle $\pi_\mu$. In order to do this, we will assume that the discrete Lagrangian $L_d$ satisfies an additional regularity condition which is usually defined as follows.

\begin{definition}
Given a symmetry group $H$ of a DMS $(Q,L_d)$, it is said that $L_d$ is \emph{$H$-regular} in $(q_0,q_1)\in Q\times Q$ if the restriction of its associated bilinear form $D_2 D_1 L_d (q_0,q_1): T_{q_0}Q\times T_{q_1}Q \lra \rr$ to $\mathcal{V}^H_{q_0}\times \mathcal{V}^H_{q_1}$ is nondegenerate.
\end{definition}

We are going to define a certain type of symmetries that guarantee that the submanifold $J_\mu^{-1}(\mu)$ defines an affine discrete connection on the principal bundle $\pi_{\mu}$.

\begin{definition}\label{def:mu_good_symmetry}
Let $G_\mu$ be a symmetry group of a DMS $(Q,L_d)$ such that the discrete Lagrangian $L_d$ is regular and $G_\mu$-regular. We say that $G_\mu$ is a \emph{group of $\mu$-good symmetries} if the following conditions are satisfied.
\begin{enumerate}
	\item For every $q \in Q$ there exists a unique $g \in G_\mu$ such that $J_\mu(q,l^Q_g(q)) = \mu$.

	\item If $J_d(q_0,q_1) = \mu = J_d(q_0,l^Q_g(q_1))$ for any $(q_0,q_1) \in Q \times Q$ and $g \in G_\mu$, then $g = e$.
\end{enumerate}

In this case, we define $\gamma(q) := g$ (so that $\widetilde{\gamma}(q) := l^Q_g(q)$).
\end{definition}

\begin{remark}
The definition of $\mu$-good symmetry given in Definition 11.10 of \cite{ar:fernandez_tori_zuccalli:2010:lagrangian_reduction_of_discrete_mechanical_systems} does not require condition (2) above. The addition is related to what we discussed in Remark \ref{rmk:differences_in_definitions_discrete_connections}.
\end{remark}

Proposition 11.11 of \cite{ar:fernandez_tori_zuccalli:2010:lagrangian_reduction_of_discrete_mechanical_systems} provides a proof of the following result\footnote{Of course, the proof of Proposition 11.11 of \cite{ar:fernandez_tori_zuccalli:2010:lagrangian_reduction_of_discrete_mechanical_systems} is based on the notion of affine discrete connection used in that paper. Still, adding the extra condition (2) from Definition \ref{def:mu_good_symmetry}, the thesis of the proposition remains valid for affine discrete connections in the sense of Definition \ref{def:affine_discrete_connection}.}, which implies that the $\mu$-good symmetries condition guarantees the existence of an affine discrete connection on the principal bundle $\pi_{\mu}$ . 

\begin{proposition}\label{prop:mu_good_imp_aff_conn}
Let $G_\mu$ be a group of $\mu$-good symmetries of $(Q,L_d)$. Then,
for every $(q_0,q_1)\in Q \times Q$ there exists a unique $g\in G_\mu$ such that $ J_\mu(q_0,l^Q_{g^{-1}}q_1)=\mu$. Then, $\calJ_\mu := J_\mu^{-1}(\{ \mu \})$ is an affine discrete connection in $\pi_\mu : Q \lra Q/G_\mu$, with level $\gamma$ (given in Definition \ref{def:mu_good_symmetry}).
\end{proposition}

\begin{definition}
Given $\mu\in \frak g^*$, if $G_\mu$ is a group of $\mu$-good symmetries of $(Q,L_d)$, we define the map $\mathcal{A}_{\mu}:Q\times Q\lra G_{\mu}$ by $\mathcal{A}_{\mu}(q_0,q_1):=g$, where $g$ is the element of $G_\mu$ appearing in the previous result. This function is called \emph{affine discrete connection associated to the momentum map $J_\mu$} and its horizontal submanifold $\Hor_{\mathcal{A}_{\mu}}$ is $\calJ_\mu$.
\end{definition}

As mentioned at the beginning of this section, those trajectories contained in the horizontal submanifold of an affine discrete connection ${\mathcal{A}_{\mu}}$ for some $\mu \in  \frak g^*$ give rise to reduced trajectories of a DMS. We now perform the reduction procedure described in the previous section to remove the symmetry corresponding to the group $G_\mu$, obtaining a DMS on the reduced space $Q/G_\mu \times Q/G_\mu$.

\subsubsection{Reduced spaces, Lagrangians and forces}

In order to carry out the Routh reduction in the setting presented in the previous section, we will consider the spaces and diffeomorphisms defined before but associated to the action of the isotropy subgroup $G_\mu$, instead of $G$.

So, we consider a principal connection $\frakA$ on the principal bundle $\pi_{\mu}$ with horizontal subbundle $\Hor_\frakA$. Hence, $TQ = \mathcal{V}^{G_{\mu}} \oplus \Hor_\frak A$, where the \emph{vertical subbundle} $\mathcal{V}^{G_{\mu}}$ over $Q$ is the subbundle of $TQ$ whose fibers are $\mathcal{V}^{G_{\mu}}_q:=T_q(l^Q_{G_{\mu}}(\{q\}))$. The \emph{horizontal lift} associated to $\frakA$ is the function $h_{\mu} : Q \times_{Q/G_{\mu}} T(Q/G_{\mu}) \lra TQ$ defined by
$$h_{\mu}^q(w_{\pi(q)}) := v_q \ \text{if} \ v_q \in \Hor_\frak A(q) \ \text{and} \ T_q\pi_{\mu} ( v_q )= w_{\pi(q)}.$$

Considering the action of $G_{\mu}$ on $Q \times G_{\mu}$ given by $l_g^{Q \times G_{\mu}}(q,w) := (l_g^Q(q),l_g^{G_{\mu}}(w))$, with $l_g^{G_{\mu}}(w) := g w g^{-1}$, we define the quotient manifold $\widetilde{G_{\mu}} := (Q \times G_{\mu})/G_{\mu}$ and the canonical projection $\rho_{\mu} : Q \times G_{\mu} \lra \widetilde{G_{\mu}}$. Then, $\widetilde{G_{\mu}}$ is a fiber bundle over $Q/G_{\mu}$ called \emph{conjugated associated bundle}, whose projection is denoted $\pr_{Q/G_{\mu}}$.

Using the principal connection $\frakA$ and the affine discrete connection ${\mathcal{A}_{\mu}}$, we can apply the ideas of Section \ref{subsubsec:reduced_spaces_and_lagrangians_general} and obtain that the quotient manifold $(\Hor_{\mathcal{A}_{\mu}})/G_{\mu}$ and the product manifold $Q/G_\mu \times Q/G_\mu$ are diffeomorphic, a computation that can be seen in detail in \cite{ar:caruso_fernandez_tori_zuccalli:2026:remarks_on_structures_and_preservation_in_forced_discrete_mechanical_systems_of_Routh_type}. Therefore, we can consider the discrete Lagrangian $\displaystyle{\breve{L}_{\mu}} : Q/G_\mu \times Q/G_\mu \lra \rr$ and we will see that the dynamics of the reduced system corresponds to that of a system with external forces, a kind of system that will be studied in more detail in Section \ref{section:forced systems}.



\subsubsection{The reduction theorem}

Let us restate the general reduction theorem in terms of this $G_{\mu}$-symmetry and the affine discrete connection $\calA_\mu$ associated to the momentum map. To this end, we need a few lemmas that show how the statement of Theorem \ref{th:4_points-general} is modified with this choice of symmetry group and affine discrete connection. The proofs of these results can be found in Lemmas 10.4 and 10.5 of \cite{ar:fernandez_tori_zuccalli:2010:lagrangian_reduction_of_discrete_mechanical_systems}.

\begin{lemma}
Let $\calA_\mu$ be the affine discrete connection associated to the momentum map $J_\mu$, $\tau_0,\tau_1 \in Q/G$ and $q_0 \in \pi_\mu^{-1}(\tau_0)$. Then,
\begin{equation*}
d\breve{L}_\mu(\tau_0,\tau_1)(\delta \tau_0,\delta \tau_1) = D_1\check{L}_d(q_0,e,\tau_1)(h_\mu^{q_0}(\delta \tau_0)) + D_3\check{L}_d(q_0,e,\tau_1)(\delta \tau_1),
\end{equation*}
for every $(\delta \tau_0, \delta \tau_1) \in T_{(\tau_0,\tau_1)}(Q/G \times Q/G)$.
\end{lemma}

Since the submanifold $\calJ_\mu$ is diffeomorphic to $Q \times \{ e \} \times (Q/G_\mu)$ via $\tilde{\Phi}_{\calA_\mu}$, the relation between the infinitesimal variations $\delta q.$ and $(\delta v.,\delta \tau.)$ considered in Proposition \ref{prop:q trajectory iff vtau satisfies red var principle} becomes simpler. Explicitly, the elements $w_k$ of the Lie group $G$ are all equal to $e$, and so the terms involving the infinitesimal generators appearing in the definition of $\delta v_k$ in equation \eqref{eq:delta_vk-def} disappear. This allows us to express the action $\check{S}_d(v.,\tau.)$ entirely in terms of elements of $Q/G_\mu$.

\begin{lemma}
Let $(v.,\tau.)$ be a discrete curve on $\tilde{G} \times (Q/G)$ such that $v_k = \rho(q_k,e)$ for some $q_k \in Q$ and all $k$. Then, for all infinitesimal variations $\delta \tau.$ with vanishing endpoints and
\[
\delta v_k := T_{(q_k,e)}\rho \left( h_\mu^{q_k}(\delta \tau_k) , T_{(q_k,q_{k+1})}\calA_\mu(h_\mu^{q_k}(\delta \tau_k) , h_\mu^{q_{k+1}}(\delta \tau_{k+1})) \right)
\]
for all $k$, we have
\begin{equation*}
d\check{S}_d(v.,\tau.)(\delta v.,\delta \tau.) = d\breve{S}_d(\tau.)(\delta \tau.) + \sum_{k=1}^{N-1} \left( \breve{f}_\mu^-(\tau_k,\tau_{k+1}) + \breve{f}_\mu^+(\tau_{k-1},\tau_k) \right) (\delta \tau_k),
\end{equation*}
where
$$\breve{f}_\mu(\tau_0,\tau_1)(\delta \tau_0,\delta \tau_1) := D_2\breve{L}_\mu (q_0,e,\tau_1) \left( T_{(q_0,q_1)}\calA_\mu \left( h_\mu^{q_0}(\delta \tau_0) , h_\mu^{q_1}(\delta \tau_1) \right) \right).$$
\end{lemma}

\begin{remark}
Since $\breve{f}_\mu$ is a $1$-form on $Q/G \times Q/G$, the terms involving $\breve{f}_\mu$ in the previous lemma can be interpreted as force terms, a concept that we will discuss in detail in Section \ref{section:forced systems}.
\end{remark}

This leads to the following version of Theorem \ref{th:4_points-general}, which is Theorem 3.15 of \cite{ar:caruso_fernandez_tori_zuccalli:2026:remarks_on_structures_and_preservation_in_forced_discrete_mechanical_systems_of_Routh_type}.

\begin{theorem}\label{hor-sym-theor-mu}
Let $G_{\mu}$ be a group of $\mu$-good symmetries of a DMS $(Q,L_d)$ for some $\mu \in \frak g^{*}$. Let $q. := (q_0,\ldots,q_N)$ be a discrete curve on $Q$ with momentum $\mu$ and let $\tau_k := \pi_\mu(q_k)$ be the corresponding reduced curve on $Q/G_\mu$. Let $\calA_\mu$ be the affine discrete connection associated to the momentum map $J_\mu$ and $h_\mu$ be the horizontal lift associated to a principal connection $\frakA$ on the principal bundle $\pi_\mu : \calJ_\mu \lra \calJ_\mu/G_\mu$. Then, the following statements are equivalent:
\begin{enumerate}
\item $q.$ satisfies the variational principle $dS_d(q.)(\delta q.) = 0$ for every infinitesimal variation $\delta q.$ with fixed endpoints.

\item $q.$ satisfies the discrete Euler--Lagrange equations \eqref{dELe}.
		
\item $\tau.$ satisfies the variational principle
\begin{equation*}
d\breve{S}_\mu(\tau.)(\delta \tau.) = - \sum_{k=0}^{N-1} \breve{f}_\mu(\tau_k,\tau_{k+1}) (\delta \tau_k, \delta \tau_{k+1}),
\end{equation*}
for every infinitesimal variations $\delta \tau.$ with fixed endpoints, where
$$\breve{S}_\mu(\tau.) := \sum_{k=0}^{N-1} \breve{L}_\mu(\tau_k,\tau_{k+1})$$
and
\[
\breve{f}_\mu(\tau_0,\tau_1)(\delta \tau_0,\delta \tau_1) := D_2\breve{L}_\mu (q_0,e,\tau_1) \left( T_{(q_0,q_1)}\calA_\mu \left( h_\mu^{q_0}(\delta \tau_0) , h_\mu^{q_1}(\delta \tau_1) \right) \right).
\]

\item $\tau.$ satisfies the forced discrete Euler--Lagrange equations
\begin{equation}\label{fdeEL}
D_1 \breve{L}_\mu (\tau_k,\tau_{k+1}) + D_2 \breve{L}_\mu (\tau_{k-1},\tau_k) + \breve{f}_\mu^+(\tau_{k-1},\tau_k) + \breve{f}_\mu^-(\tau_k,\tau_{k+1}) = 0,
\end{equation}
for $k = 1,\ldots,N-1$.
\end{enumerate}
\end{theorem}

\section{Reduction of symmetric forced discrete mechanical systems}\label{section:forced systems}

In this section, inspired by the presence of forces in the equations of motion of the dynamical system obtained after a process of discrete Routh reduction (see Theorem \ref{hor-sym-theor-mu}), we turn our attention to discrete mechanical systems with external forces. We first give the necessary definitions and then study their symmetry reduction process, as developed in \cite{ar:caruso_fernandez_tori_zuccalli:2023:lagrangian_reduction_of_forced_discrete_mechanical_systems}.

\subsection{Forced discrete mechanical systems and their dynamics}

\begin{definition}
A \emph{forced discrete mechanical system} (FDMS) is a triple $(Q,L_d,f_d)$ where $Q$ and $L_d$ are as in Definition \ref{def:dms}, and $f_d$ is a $1$-form on $Q \times Q$, the \emph{discrete force}.
\end{definition}

Using the identification we mentioned previously (see Section \ref{subsection:preliminaries_on_product_manifolds}), the force $f_d \in \Gamma(Q \times Q, T^{*}(Q\times Q))$ can be decomposed as $f_{d}^{-} \oplus f_{d}^{+}$, with $f_{d}^{-} \in \Gamma(Q \times Q, \pr_{1}^{*} (T^{*}Q))$ and
$f_{d}^{+} \in \Gamma(Q \times Q, \pr_{2}^{*} (T^{*}Q))$. Then,
$f_d^- (q_0,q_1) \in T_{q_0}^*Q$, $f_d^+ (q_0,q_1) \in T_{q_1}^*Q$
and
$$f_d(q_0,q_1)(\delta q_0,\delta q_1) = f_d^- (q_0,q_1)(\delta q_0) + f_d^+ (q_0,q_1)(\delta q_1).$$

The \emph{discrete action} of a FDMS $(Q,L_d,f_d)$ is defined as in the case of a DMS, while its dynamics is given by an appropriate modification of the discrete Hamilton principle, known as discrete Lagrange--d'Alembert principle.

\begin{definition}
A discrete curve $q.$ is a \emph{trajectory} of $(Q,L_d,f_d)$ if it satisfies
\begin{equation}\label{eq:discrete_lagrange_dalembert_principle}
dS_d(q.)(\delta q.) + \sum_{k=0}^{N-1} f_d(q_k,q_{k+1})(\delta q_k,\delta q_{k+1}) = 0,
\end{equation}
for every infinitesimal variation $\delta q.$ of $q.$ with fixed endpoints. 
\end{definition}

The trajectories of the system can be characterized in terms of equations as follows:

\begin{theorem}\label{CVForzado}
Let $(Q,L_d,f_d)$ be a FDMS. Then, a discrete curve $q. : \{ 0,\ldots,N \} \lra Q$ is a trajectory of $(Q,L_d,f_d)$ if and only if it satisfies the following algebraic equations:
\begin{equation}\label{forcedELe}
D_2 L_d(q_{k-1},q_k) + D_1 L_d(q_k,q_{k+1}) + f_d^+(q_{k-1},q_k) + f_d^-(q_k,q_{k+1}) = 0 \in T_{q_k}^{*}Q,
\end{equation}
for all $k=1,\ldots,N-1$, called \emph{forced discrete Euler--Lagrange equations}.
\end{theorem}
\begin{proof}
See \cite[Part III]{ar:marsden_west:2001:discrete_mechanics_and_variational_integrators}.
\end{proof}
 
\subsection{Forced discrete Legendre transforms and regularity}

As in the unforced case, we have two fiber derivatives for $(Q,L_d,f_d)$.

\begin{definition}
Let $(Q,L_d,f_d)$ be a FMDS. The maps $\ff^+_{f_d}L_{d}: Q \times Q \lra T^*Q$ and $\ff^-_{f_d}L_{d} : Q \times Q \lra T^*Q$ defined by
\[
\begin{split}
&\ff^+_{f_d}L_{d}(q_0,q_1) := D_2 L_d(q_0,q_1) + f_d^+(q_0,q_1), \\
& \ff^-_{f_d}L_{d}(q_0,q_1) := -D_1 L_d(q_0,q_1) - f_d^-(q_0,q_1)
\end{split}
\]
are called \emph{forced discrete Legendre transforms}.
\end{definition}

\begin{remark}\label{preservation}
These maps, $\ff^-_{f_d}L_{d}$ and $\ff^+_{f_d}L_{d}$, preserve the base points of the fiber bundles $\pr_{i}:Q\times Q \lra Q$ and $\pi_Q : T^*Q \lra Q$, $i=1,2$, respectively. That is,
$$\ff^-_{f_d}L_{d}(q_0,q_1) \in T^{*}_{q_0}Q \ \ \mbox{and} \ \ \ff^+_{f_d}L_{d}(q_0,q_1) \in T^{*}_{q_1}Q.$$

Hence, $\ff^-_{f_d}L_{d} \in \Gamma(Q \times Q, \pr_{1}^{*} (T^{*}Q))$ and 
$\ff^+_{f_d}L_{d} \in \Gamma(Q \times Q, \pr_{2}^{*} (T^{*}Q))$.
\end{remark}

\begin{remark}
Once again, it is worth noting that the forced discrete Euler--Lagrange equations \eqref{forcedELe} can be written as $\ff^+_{f_d}L_{d}(q_{k-1},q_{k}) =\ff^-_{f_d}L_{d}(q_{k},q_{k+1})$ for all $k=1,...,N-1$.
\end{remark}

\begin{definition}
A FDMS $(Q,L_d,f_d)$ is said to be \emph{regular} if the forced discrete Legendre transforms $\ff^+_{f_d}L_{d}$ and $\ff^-_{f_d}L_{d}$ are local isomorphisms of fiber bundles or, equivalently, if they are local diffeomorphisms. That is, for every $q_0, q_1 \in Q$, the maps
$$\phi_{q_1}:Q\lra T^{*}_{q_1}Q, \quad \phi_{q_1}(q_0):=\ff^+_{f_d}L_{d}(q_0,q_1)=D_2 L_d(q_0,q_1) + f_d^+(q_0,q_1) $$
$$\phi_{q_0}:Q\lra T^{*}_{q_0}Q, \quad \phi_{q_0}(q_1):=\ff^-_{f_d}L_{d}(q_0,q_1)=-D_1 L_d(q_0,q_1) - f_d^-(q_0,q_1)$$
are local diffeomorphisms.
\end{definition}

Just like in the unforced case, the time evolution of the system $(Q,L_d,f_d)$ can be described by a function taking values in $Q\times Q$ and whose domain is an open subset of $Q\times Q$, called \emph{forced discrete flow}. This function is not always well--defined, but the regularity of the system can be used to guarantee its existence (see \cite[Theorem 2.11]{ar:caruso_fernandez_tori_zuccalli:2026:remarks_on_structures_and_preservation_in_forced_discrete_mechanical_systems_of_Routh_type}).

\begin{theorem}\label{equations}
Let $(Q,L_d,f_d)$ be a regular FDMS. Given a trajectory $(q_0,q_1,q_2)$ of $(Q,L_d,f_d)$, there exist two open subsets $U, V \subs Q \times Q$ and a diffeomorphism ${\bf F}_{L_d,f_d} : U \lra V$ such that
\begin{enumerate}
\item $(q_0,q_1) \in U, (q_1,q_2) \in V$ and ${\bf F}_{L_d,f_d}(q_0, q_1) = (q_1,q_2)$.
\item For every $(\tilde{q_0},\tilde{q_1}) \in U$, if $\tilde{q_2} := \pr_2({\bf F}_{L_d,f_d}(q_0, q_1))$, then $(\tilde{q_0}, \tilde{q_1}, \tilde{q_2})$ is a trajectory of $(Q,L_d,f_d)$.
\item Every trajectory $(q'_0, q'_1, q'_2 )$ of $(Q,L_d,f_d)$ such that $(q'_0, q'_1) \in U$ and $(q'_1, q'_2) \in V$ satisfies
$(q'_1, q'_2) = {\bf F}_{L_d,f_d}(q'_0, q'_1)$.
\end{enumerate}
\end{theorem}

\subsection{Reduction of symmetric forced discrete mechanical systems}

Following the ideas found in \cite{ar:caruso_fernandez_tori_zuccalli:2023:lagrangian_reduction_of_forced_discrete_mechanical_systems}, we establish a reduction procedure for discrete mechanical systems with external forces. We begin by recalling the definition of forced discrete Lagrangian system with symmetries.

\begin{definition}\label{sym-group-def}
Let $(Q,L_d,f_d)$ be a FDMS and let $G$ be a Lie group acting on the left on $Q$ such that $\pi : Q \lra Q/G$ is a principal bundle with structure group $G$. Considering the diagonal action of $G$ on $Q \times Q$, we say that $G$ is a \emph{symmetry group} of $(Q,L_d,f_d)$ if $L_d$ is $G$-invariant and $f_d : Q \times Q \lra T^*(Q \times Q)$ is $G$-equivariant, considering the cotangent lift $l^{T^*(Q \times Q)}$ of the diagonal action on $T^*(Q \times Q)$\footnote{The $G$-equivariance of $f_d$ as a section is equivalent to its $G$-invariance as a $1$-form.}.
\end{definition}

\begin{remark}
As seen in \cite{ar:caruso_fernandez_tori_zuccalli:2023:lagrangian_reduction_of_forced_discrete_mechanical_systems}, we can define a momentum map associated to a symmetry and study its time evolution along the trajectories of the system.
\end{remark}

\subsubsection{Reduced spaces, Lagrangians and forces}

The process of symmetry reduction of a forced discrete mechanical system was developed in \cite{ar:caruso_fernandez_tori_zuccalli:2023:lagrangian_reduction_of_forced_discrete_mechanical_systems} following the general ideas ideas of reduction of discrete mechanical systems. The considered diffeomorphic model for the quotient manifold $(Q\times Q)/G$ is $\tilde{G} \times (Q/G)$ and the product reduced Lagrangian is still $\displaystyle{\hat{L}_d:\tilde{G} \times (Q/G)\lra \rr}$. In that article, the authors studied in detail how to consider a reduced version of the discrete force, using the following result.

\begin{proposition}
Let $\mathcal{A}_d$ be an affine discrete connection on $\pi : Q \lra Q/G$ and $\tilde{\frak A}$ be a principal connection on $\tilde{\pi} : Q \times Q \lra (Q \times Q)/G$. There exists an isomorphism of vector bundles over $\tilde{G} \times (Q/G)$
$$\Psi_{{\calA}_{d}}^*(T^*(Q \times Q)/G) \simeq T^*(\tilde{G} \times Q/G) \oplus \hat{\frakg}^*,$$
where $\Psi_{\calA_d} := (\Phi_{\calA_d})^{-1}$, $\hat{\frakg} := \Psi_{{\calA}_{d}}^* \tilde{\frakg}$ and $\tilde{\frakg} := ((Q \times Q \times \frakg)/G)$ is the adjoint bundle, with $G$ acting on $\frakg$ by the adjoint action.
\end{proposition}

\begin{corollary}
Considering an affine discrete connection $\mathcal{A}_d$ on the principal bundle $\pi : Q \lra Q/G$ and a principal connection $\tilde{\frak A}$ on the principal bundle $\tilde{\pi} : Q \times Q \lra (Q \times Q)/G$, a $G$-equivariant section $f_d : Q \times Q \lra T^*(Q \times Q)$ induces a section $\hat{f}_d : \tilde{G} \times (Q/G) \lra T^*(\tilde{G} \times (Q/G)) \oplus \hat{\frakg}^*$. Explicitly,
$$\hat{f}_d(v_0,\tau_1)((\delta v_0,\delta \tau_1) \oplus [(q_0,q_1),\xi]) = f_d(q_0,q_1)(\delta q_0,\delta q_1),$$
for $(\delta q_0,\delta q_1) \in T(Q \times Q)$ such that $(\delta v_0,\delta \tau_1) = T_{(q_0,q_1)}\Upsilon (\delta q_0,\delta q_1)$ and $\xi = \tilde{\frak A}(\delta q_0,\delta q_1)$.
\end{corollary}

\subsubsection{Reduced dynamics}

As is customary when establishing a reduction process for a $G$-symmetry, let us consider a principal connection $\frakA$ and an affine discrete connection $\mathcal{A}_d$ on the principal bundle $\pi : Q \lra Q/G$.

From the principal connection $\frakA$, we will consider the principal connection $\tilde{\frak A}$ on $\tilde{\pi} : Q \times Q \lra (Q \times Q)/G$, induced in a natural way, given by $\tilde{\frak A} := \frac{1}{2} (\pr_1^* \frak A + \pr_2^* \frak A)$. Explicitly,
\[
\tilde{\frak A}(\delta q_0,\delta q_1) := \frac{1}{2} \left( \frak A(\delta q_0) + \frak A(\delta q_1) \right).
\]

Just like it is necessary to consider the version of the Lagrangian defined on the space $Q \times G \times (Q/G)$ to describe the dynamics of the reduced system, we will consider the following version of the discrete force:
$$\check{f}_d := \tilde{\Psi}_{{\calA}_{d}}^* f_d \in \Omega^{1}(Q \times G \times Q/G),$$
which can decomposed as
\begin{equation}\label{fd check}
\begin{split}
\check{f}_d(q_0,w_0,\tau_1)(\delta q_0,\delta w_0,\delta \tau_1) &= \check{f}_d^1(q_0,w_0,\tau_1)(\delta q_0) + \check{f}_d^2(q_0,w_0,\tau_1)(\delta w_0) \\
& \ \ \ \ \ + \check{f}_d^3(q_0,w_0,\tau_1)(\delta \tau_1)
\end{split}
\end{equation}
for all $(\delta q_0,\delta w_0,\delta \tau_1) \in T_{(q_0,w_0,\tau_1)}(Q \times G \times Q/G)$.

\begin{theorem}\label{teor-red-forzada}
Let $q. = (q_0,\ldots,q_N)$ be a discrete curve in $Q$ and let
\begin{equation*}
\begin{split}
\tau_k &:= \pi(q_k), \quad 1 \le k \le N, \\
w_k &:= \mathcal{A}_d(q_k,q_{k+1}), \quad v_k := \rho(q_k,w_k), \quad 0 \le k \le N-1
\end{split}
\end{equation*}
be the corresponding discrete curves in $Q/G$, $G$ and $\tilde{G}$. Then, given a forced discrete mechanical system $(Q,L_d,f_d)$ with symmetry group $G$, the following statements are equivalent.
\begin{enumerate}
\item $q.$ satisfies the variational principle \eqref{eq:discrete_lagrange_dalembert_principle}.

\item $q.$ satisfies the forced discrete Euler--Lagrange equations \eqref{forcedELe}.

\item $(v_\cdot,\tau_\cdot) = ((v_0,\tau_1),\ldots,(v_{N-1},\tau_N))$ satisfies
\begin{equation*}
\begin{split}
\delta &\left( \sum_{k=0}^{N-1} \hat{L}_d(v_k,\tau_{k+1}) \right) \\
& \ \ \ \ \ + \sum_{k=0}^{N-1} \hat{f}_d(v_k,\tau_{k+1}) \left( (\delta v_k,\delta \tau_{k+1}) \oplus \left[ (q_k,q_{k+1}) , \frac{1}{2}(\xi_k + \xi_{k+1}) \right] \right) = 0
\end{split}
\end{equation*}
(3.2) for all infinitesimal variations $(\delta v_\cdot, \delta \tau_\cdot,\xi_\cdot)$ satisfying:
\begin{enumerate}
\item $\delta \tau_k \in T_{\tau_k} (Q/G)$ for $k=1,\ldots,N$,
	
\item $\xi_\cdot =(\xi_0,\ldots,\xi_N)\in \frak{g}^{N+1}$,
	
\item for $k=0,\ldots,N-1$,
\[
\begin{split}
\delta v_k &= T_{(q_k,w_k)}\rho(h^{q_k}(\delta \tau_k),T_{(q_k,q_{k+1})}\mathcal{A}_{d}(h^{q_k}(\delta
\tau_k),h^{q_{k+1}}(\delta \tau_{k+1}))) \\
& \ \ \ +
T_{(q_k,w_k)}\rho ((\xi_k)_Q(q_k),
T_{(q_k,q_{k+1})}\mathcal{A}_d ((\xi_k)_Q(q_k), (\xi_{k+1})_Q(q_{k+1})))
\end{split}
\]
where $h$ is the horizontal lift associated to $\frak A$, $\delta \tau_{0} \in T_{\pi(q_0)}(Q/G)$,
	
\item and the ``fixed endpoint conditions": 
$\delta \tau_0 = 0$ \hspace{.1pc}\footnote{Notice that $\delta {\tau}_{0}$ is just a convenient device to simplify the expression of the $\delta v_k$.}, $\delta \tau_N=0$, $\xi_0=\xi_N=0$.
\end{enumerate}

\item $(v.,\tau.)$ satisfies the following conditions for each fixed $(v_{k-1},\tau_k,v_k,\tau_{k+1})$, with $1 \le k \le N-1$.
	\end{enumerate}	
\begin{itemize}
\item $\phi_{f_d} = 0$ for $\phi_{f_d} \in T_{\tau_k}^*(Q/G)$ defined by
\begin{equation}\label{red-traj-eq1}
	\begin{split}
		\phi_{f_d} &:= D_1 \check{L}_d (q_k,w_k,\tau_{k+1}) \circ h^{q_k} + D_3 \check{L}_d (q_{k-1},w_{k-1},\tau_k) \\
		&+ D_2 \check{L}_d (q_k,w_k,\tau_{k+1}) D_1\mathcal{A}_d(q_k,q_{k+1}) \circ h^{q_k} \\
		&+ D_2 \check{L}_d (q_{k-1},w_{k-1},\tau_k) D_2\mathcal{A}_d(q_{k-1},q_k) \circ h^{q_k} \\
		&+ \check{f}_d^1(q_k,w_k,\tau_{k+1}) \circ h^{q_k} + \check{f}_d^3(q_{k-1},w_{k-1},\tau_k) \\
		&+ \check{f}_d^2(q_k,w_k,\tau_{k+1}) D_1\mathcal{A}_d(q_k,q_{k+1}) \circ h^{q_k} \\ 
		&+ \check{f}_d^2(q_{k-1},w_{k-1},\tau_k) D_2\mathcal{A}_d(q_{k-1},q_k) \circ h^{q_k}.
	\end{split}
\end{equation}

\item $\psi_{f_d}(\xi_k) = 0$ for all $\xi_k \in \frakg$, where $\psi_{f_d} \in \frakg^*$ is defined by
\begin{equation}\label{red-traj-eq2}
	\begin{split}
		\psi_{f_d}(\xi_k) &:= \left( D_2 \check{L}_d (q_{k-1},w_{k-1},\tau_k) w_{k-1}^{-1} - D_2 \check{L}_d (q_k,w_k,\tau_{k+1}) w_k^{-1} \right) (\xi_k) \\
		&+ \left( \check{f}_d^1(q_k,w_k,\tau_{k+1}) + \check{f}_d^2(q_k,w_k,\tau_{k+1}) D_1 \mathcal{A}_d(q_k,q_{k+1}) \right. \\
		&+ \left. \check{f}_d^2(q_{k-1},w_{k-1},\tau_k) D_2 \mathcal{A}_d (q_{k-1},q_k) \right) \left( (\xi_k)_Q(q_k) \right),
	\end{split}
\end{equation}
where we are using the notation of Remark \ref{remark:notation_lie_groups}.
\end{itemize}
\end{theorem}

\begin{remark}
If $f_d \equiv 0$, the statement of Theorem \ref{teor-red-forzada} reduces to the unconstrained version of Theorem 5.11 of \cite{ar:fernandez_tori_zuccalli:2010:lagrangian_reduction_of_discrete_mechanical_systems}.
\end{remark}

\section{Reduction of symmetric nonholonomic discrete mechanical systems}\label{section:nonholonomic_reduction}

In this section we describe discrete mechanical systems with nonholonomic constraints (nHDMS). These systems can be considered as the well--known generalized nonholonomic systems, since the variational constraints are not deduced from the kinematical ones. In these systems, the kinematic constraints are given by a non integrable (in the sense of Frobenius) submanifold of $Q \times Q$ and the variational constraints are described by a subbundle of $TQ$.

\subsection{Nonholonomic discrete mechanical systems and their dynamics}

\begin{definition}
A \emph{nonholonomic discrete mechanical system} consists of a set of objects $\left( Q,L_{d},\mathcal{D},\mathcal{D}_{d}\right) $, where $Q$ and $L_d$ are as in Definition \ref{def:dms}, $\mathcal{D}$ is a subbundle of $TQ$ and 
$\mathcal{D}_{d}$ is a submanifold of $Q\times Q$.

The manifold $Q$ is the configuration space, the function $L_{d}$ is the discrete Lagrangian, $\mathcal{D}$ is the \emph{space of variational constraints} and $\mathcal{D}_{d}$ is the space of \emph{discrete kinematical constraints}.
\end{definition}

\begin{remark}
In the context of the previous definition, if the distribution $%
\mathcal{D}$ and the constraints submanifold $\mathcal{D}_{d}$ satisfy $\mathcal{D}=TN$ and $\mathcal{D}_{d}=N\times N$, with $N \subs Q$ a submanifold of $Q$, we have a \emph{holonomic discrete mechanical system}. That is, the constraints are given by the tangent bundle of a submanifold of the configuration space. In this case, it can be seen that the equations of motion of a discrete mechanical system with holonomic constraints are equivalent to the equations of motion of a (free) DMS defined on $N$.
\end{remark}

The \emph{discrete Lagrange--d'Alembert principle} establishes that a trajectory of the system $\left( Q,L_{d},\mathcal{D},\mathcal{D}_{d}\right) $ is a discrete curve $q.$ such that
\begin{enumerate}
\item \label{Marker_DLA_variacional} is a critical point of $S_{d}$ for every admissible variation; i.e.,
\begin{equation*}
dS_{d}\left( q.\right) \left( \delta q.\right) =0
\end{equation*}%
for every infinitesimal variation $\delta q.$ with fixed endpoints such that $\delta
q_{k}\in \mathcal{D}_{q_{k}}$ for all $k$.

\item $q.$ satisfies the kinematical constraints; i.e., $\left(
q_{k},q_{k+1}\right) \in \mathcal{D}_{d}$ for all $k$.
\end{enumerate}

Just like the discrete Hamilton principle, this one gives rise to a set of equations. We have that $q.$ is a trajectory of the nHDMS $\left(
Q,L_{d},\mathcal{D},\mathcal{D}_{d}\right) $ if and only if for all $k$%
\begin{equation}
D_{1}L_{d}\left( q_{k},q_{k+1}\right) +D_{2}L_{d}\left( q_{k-1},q_{k}\right)
\in \mathcal{D}_{q_{k}}^{\circ }\text{ and }\left( q_{k},q_{k+1}\right) \in 
\mathcal{D}_{d}\text{.}  \label{EqDLA_eq1}
\end{equation}%

Analogously, if $\mathcal{D}_{q_{k}}^{\circ }=\left\langle \omega
^{a}\left( q_{k}\right) \right\rangle $ for certain $1$-forms $\omega ^{a}$ on $Q$ and $\mathcal{D}_{d}$ is determined by local equations $%
\chi _{K}\left( q_{k},q_{k+1}\right) =0$, $q.$ is a trajectory of $\left( Q,L_{d},\mathcal{D},\mathcal{D}_{d}\right) $ if and only if for all $k$ the \emph{discrete Lagrange--d'Alembert equations}
\begin{equation}
\left\{ 
\begin{array}{l}
D_{1}L_{d}\left( q_{k},q_{k+1}\right) +D_{2}L_{d}\left( q_{k-1},q_{k}\right)
=\lambda _{a,k}\omega ^{a}\left( q_{k}\right) \medskip \\ 
\chi _{K}\left( q_{k},q_{k+1}\right) =0%
\end{array}%
\right.  \label{EqDLA_ecuaciones}
\end{equation}%
are verified, where $\lambda _{a,k}\in \rr$ are constants. Under certain conditions, it can be proven that the discrete Lagrange--d'Alembert equations generate a well--defined diffeomorphism $\mathbf{F}^{nH}_{L_{d}}:\mathcal{D}
_{d}\lra \mathcal{D}_{d}$ (see Theorem \ref{thm:nHflow_well_defined} below).

\subsection{Regularity and existence of flow}

In this section, we state a result of McLachlan and Perlmutter \cite{ar:mclachlan_perlmutter:2006:integrators_for_nonholonomic_mechanical_systems} that establishes certain regularity conditions that are equivalent to the ones mentioned by Cort\'{e}s and Mart\'{\i}nez in \cite{ar:cortes_martinez:2001:non_holonomic_integrators} and guarantee the existence and uniqueness of solution for the set of equations \eqref{EqDLA_ecuaciones}.

\begin{remark}
In the continuous setting for differential equations that determine the evolution of generalized nonholonomic mechanical systems, Balseiro and Solomin determine in \cite{ar:balseiro_solomin:2008:on_generalized_non-holonomic_systems} a condition that guarantees the existence and uniqueness of solution. This condition implies, for a certain type of constraints usually called admissible, that the dimension of the submanifold defining the kinematic constraints and the distribution $\mathcal{D}$ coincide.

Similarly, in the discrete setting, McLachlan and Perlmutter in 
\cite{ar:mclachlan_perlmutter:2006:integrators_for_nonholonomic_mechanical_systems} and Cort\'{e}s and Mart\'{\i}nez in \cite{ar:cortes_martinez:2001:non_holonomic_integrators} establish that one of the necessary hypotheses for the existence of solution of the set of equations \eqref{EqDLA_ecuaciones} is that the space of kinematic constraints have the same dimension as the space of variational constraints; i.e., $\dim \mathcal{D}_{d}=\dim \mathcal{D}$. Throughout this section, we will assume that this hypothesis holds.
\end{remark}

\begin{definition}
For each $\left( q_{0},q_{1}\right) \in Q\times Q$, let $%
\psi _{\left( q_{0},q_{1}\right) }:Q\lra \left( \mathcal{D}%
_{q_{1}}\right) ^{\ast }$ be the map defined by
\begin{equation*}
\psi _{\left( q_{0},q_{1}\right) }\left( q_{2}\right) :=\iota _{q_{1}}^{\ast
}\left( D_{1}L_{d}\left( q_{1},q_{2}\right) +D_{2}L_{d}\left(
q_{0},q_{1}\right) \right)
\end{equation*}%
where $\left( \mathcal{D}_{q_{1}}\right) ^{\ast }$ is the dual space of $%
\mathcal{D}_{q_{1}}\subs T_{q_{1}}Q$ and where $\iota _{q_{1}}^{\ast
}:T_{q_{1}}^{\ast }Q\lra \left( \mathcal{D}_{q_{1}}\right) ^{\ast }$
is the dual map of the inclusion $\iota _{q_{1}}:\mathcal{D}%
_{q_{1}}\lra T_{q_{1}}Q$.
\end{definition}

Next, we state the result that establishes the condition that guarantee the existence and uniqueness of solution for the set of equations \eqref{EqDLA_ecuaciones} (see \cite{ar:mclachlan_perlmutter:2006:integrators_for_nonholonomic_mechanical_systems}).

\begin{theorem}\label{thm:nHflow_well_defined}
Let $\left( q_{0},q_{1}\right) \in \mathcal{D}_{d}$ and let $\pr _{1}:Q\times
Q\lra Q$ be the projection onto the first factor. Assuming that $ \pr _{1}\vert _{\mathcal{D}_{d}}:\mathcal{D}_{d}\lra Q$ is a submersion, the local existence of the discrete flow $
\mathbf{F}^{nH}_{L_{d}}\left( q_{k-1},q_{k}\right) =\left( q_{k},q_{k+1}\right) $
is guaranteed whenever for each $q_{2}\in \psi _{\left(
q_{0},q_{1}\right) }^{-1}\left( 0\right) \cap \left( \left. \pr_{1}\right\vert _{\mathcal{D}_{d}}\right) ^{-1}\left( q_{1}\right) $ and each nonzero $
v_{q_{2}}\in T_{q_{2}}\left( \mathcal{D}_{d}\left( q_{1}\right) \right) $
$$\left\langle D_{2}D_{1}L_{d}\left( q_{1},q_{2}\right) \cdot
v_{q_{2}},v_{q_{1}}\right\rangle \neq 0 $$ 
for every $v_{q_{1}}\in \mathcal{D}_{q_{1}}$. When this condition holds for every $q_{1}\in Q$, the discrete Lagrange--d'Alembert equations generate a well--defined diffeomorphism $\mathbf{F}^{nH}_{L_{d}}:\mathcal{D}
_{d}\lra \mathcal{D}_{d}$.
\end{theorem}

\subsection{Symmetric nonholonomic discrete mechanical systems}

\begin{definition}
A Lie group $G$ is a \emph{symmetry group} of the system $\left( Q,L_{d},\mathcal{D},\mathcal{D}_{d}\right) $ if, in addition to $\pi :Q\lra Q/G$ being a principal bundle, $L_{d}$ and $\mathcal{D}_{d}$ are inavariant by the action $l^{Q\times Q}$ and $\mathcal{D}$ is invariant by the action $l^{TQ}$.
\end{definition}

In this section, we assume that the distribution $\mathcal{S}$ with fibers $\mathcal{S}_q := \calV^G_q \cap \calD_q$ has locally constant rank, so that it is a subbundle of $TQ$. We also assume that $G$-invariants subbundles of $TQ$, $\mathcal{W}$, $\mathcal{H}$, and $\mathcal{U}$ can be chosen in such a way that we have the following associated decomposition of the tangent bundle $TQ$,
\[
TQ = \mathcal{W} \oplus \mathcal{U} \oplus \mathcal{S} \oplus \mathcal{H},
\]
where $\calD = \mathcal{S} \oplus \mathcal{H}$ and $\calV^G = \mathcal{S} \oplus \mathcal{U}$.

This decomposition corresponds to a nonholonomic connection, with connection $1$-form $\mathcal{A}:TQ\lra \mathfrak{g}$, and associated horizontal lift $h:Q\times _{Q/G}T\left( Q/G\right) \lra TQ$.

Consider the action of $G$ on $Q\times _{Q/G}T\left( Q/G\right) $
given by $l_{g}^{Q\times _{Q/G}T\left( Q/G\right) }\left( q,w_{\pi \left( q\right)
}\right) :=\left( l_{g}^{Q}\left( q\right) ,w_{\pi \left( q\right) }\right) $.

For each $q\in Q$, the elements $\xi \in \mathfrak{g}$ satisfying $\xi
_{Q}\left( q\right) \in \mathcal{D}_{q}$ form a vector space. This vector space may be different for different values of $q$, so we consider the space $\mathfrak{g}^{\mathcal{D}}:=\left\{ \left( q,\xi \right) \in Q\times 
\mathfrak{g}:\xi _{Q}\left( q\right) \in \mathcal{D}_{q}\right\}$, which is a vector bundle over $Q$, considering the projection onto the first factor. The differential of the action $l^{Q}$ establishes an isomorphism between the fiber bundles $\mathfrak{g}^{\mathcal{D}}$ and $\mathcal{S}$, i.e., $\mathfrak{g}^{\mathcal{D}}\simeq \mathcal{S}$.

A trajectory of a nonholonomic discrete mechanical system is determined by the discrete Lagrange--d'Alembert principle. When the variational constraints are decomposed as $\mathcal{D}=\mathcal{H}\oplus 
\mathcal{S}$, it is possible to decompose the admissible variations in horizontal and vertical variations, should they arise from $\mathcal{H}$ or from $\mathcal{S}$. According to this decomposition, we can also decompose the variational principle, as stated in the following result.

\begin{proposition}
Given a discrete curve $q.$ on $Q$, the following statements are equivalent:
\begin{enumerate}
\item  
$\displaystyle{dS_{d}\left( q_{\cdot }\right) \left( \delta q_{\cdot }\right) =0}$ for every infinitesimal variation $\delta q_{\cdot }$ with fixed endpoints such that $\delta
q_{k}\in \mathcal{D}_{q_{k}}$ for all $k$.

\item 
$\displaystyle{dS_{d}\left( q_{\cdot }\right) \left( \delta q_{\cdot }^{\mathcal{H}}\right)
=0\quad \text{and}\quad dS_{d}\left( q_{\cdot }\right) \left( \delta q_{\cdot
}^{\mathcal{S}}\right) =0}$
for every pair of infinitesimal variations with fixed endpoints $\delta q_{\cdot }^{\mathcal{H}
}\in \mathcal{H}$ and $\delta q_{\cdot }^{\mathcal{S}}\in \mathcal{S} $ for all $k$.
\end{enumerate}
\end{proposition}

In the case of mechanical system with holonomic constraints, both in the continuous and discrete settings, the presence of symmetries gives rise to conserved quantities (the momenta). This is the statement of the well--known Noether's Theorem, in its continuous and discrete versions. In the case of systems with nonholonomic constraints, this is no longer case, essentially due to the constraint forces. However, in this context we can obtain an equation describing the evolution of the momentum map along the trajectories of the system.

\begin{definition}
Given a nHDMS $\left( Q,L_{d},\mathcal{D},\mathcal{D}_{d}\right) $ with symmetry group $G$, the \emph{nonholonomic discrete momentum map} is the map $J_{d}^{nh}:Q\times Q\lra \left( \mathfrak{g}^{%
\mathcal{D}}\right) ^{\ast }$ defined by
\begin{equation}
J_{d}^{nh}\left( q_{0},q_{1}\right) \left( q_{0},\xi \right) :=-D_{1}L_{d}\left(
q_{0},q_{1}\right) \xi _{Q}\left( q_{0}\right) \text{,}
\label{EqDefMomentoNoHolDisc}
\end{equation}%
for every $\left( q_{0},\xi \right) \in \mathfrak{g}^{\mathcal{D}}$. Given a section $\tilde{\xi}%
\in \Gamma \left( \mathfrak{g}^{\mathcal{D}}\right) $, we define the map $\left( J_{d}^{nh} \right) _{\tilde{\xi}}:Q\times Q\lra \rr$ by $\left( J_{d}^{nh} \right) _{\tilde{\xi}}\left( q_{0},q_{1}\right) :=J_{d}^{nh}\left(
q_{0},q_{1}\right) \left( q_{0},\tilde{\xi}\left( q_{0}\right) \right)$.
\end{definition}

The next result establishes the equivalence between the vertical variational principle and the evolution of the nonholonomic discrete momentum map.

\begin{theorem}
\label{TeoremaMomentoSinReduccion}
Let $q_{\cdot }$ be a discrete curve on $Q$. The following statements are equivalent:
\begin{enumerate}
\item $q_{\cdot }$ satisfies the vertical variational principle. That is, $
dS_{d}\left( q_{\cdot }\right) \left( \delta q_{\cdot }^{\mathcal{S}}\right)
=0$ for every infinitesimal variation with fixed endpoints such that $\delta q_{k}^{\mathcal{%
S}}\in \mathcal{S}_{q_{k}}$ for all $k$.

\item For every section $\tilde{\xi}\in \Gamma \left( \mathfrak{g}^{%
\mathcal{D}}\right) $,%
\begin{equation}\label{EqEvolucionMomentoNHDisc}
\begin{split}
\left( J_{d}^{nh} \right)_{\tilde{\xi}}\left( q_{k},q_{k+1}\right) - & \left(
J_{d}^{nh} \right) _{\tilde{\xi}}\left( q_{k-1},q_{k}\right) \\
& =-D_{1}L_{d}\left(
q_{k-1},q_{k}\right) \left( \tilde{\xi}\left( q_{k}\right) -\tilde{\xi}%
\left( q_{k-1}\right) \right) _{Q}\left( q_{k-1}\right) \text{.}
\end{split}
\end{equation}
\end{enumerate}
\end{theorem}

\subsection{Reduction of symmetric nonholonomic discrete mechanical systems}

In this section, we present the result of reduction of nonholonomic discrete mechanical systems mentioned previously.

In the following theorem we present the relation between the dynamics of a symmetric nHDMS and that of a reduced dynamical system defined using a variational principle, as well as the relation between the corresponding trajectories.

\begin{theorem}
\label{TeoremaReduccion}
Let $q.$ be a discrete curve on $Q$ and let $%
\tau_{k}:=\pi \left( q_{k}\right) $, $w_{k}:=\mathcal{A}_{d}\left(
q_{k},q_{k+1}\right) $ and $v_{k}:=\rho \left( q_{k},w_{k}\right) $ be the corresponding discrete curves on $Q/G$, $G$ and $\tilde{G}$, respectively. Then, given a nHDMS $\left( Q,L_{d},\mathcal{D},\mathcal{D}%
_{d}\right) $ with symmetry group $G$, the following statements are equivalent:
\begin{enumerate}
\item $\left( q_{k},q_{k+1}\right) \in \mathcal{D}_{d}$ for all $k$ and $%
q_{\cdot }$ satisfies the variational principle $dS_{d}\left( q_{\cdot
}\right) \left( \delta q_{\cdot }\right) =0$ for every infinitesimal variation $\delta
q_{\cdot }$ with fixed endpoints such that $\delta q_{k}\in \mathcal{D}_{q_{k}}$
for all $k$.

\item $q_{\cdot }$ satisfies the discrete Lagrange--d'Alembert equations \eqref{EqDLA_eq1} for all $k$.

\item \label{Item_3_Reduccion}$\left( v_{k},\tau_{k+1}\right) \in \mathcal{\hat{%
D}}_{d}:=\Phi _{\mathcal{A}_{d}}\left( \mathcal{D}_{d}/G\right) $ for all $%
k$ y $d\hat{S}_{d}\left( v_{\cdot },\tau_{\cdot }\right) \left( \delta v_{\cdot
},\delta \tau_{\cdot }\right) =0$ for every infinitesimal variation $\left( \delta
v_{\cdot },\delta \tau_{\cdot }\right) $ with fixed endpoints such that
\[
\delta \tau_{k}\in \mathcal{\hat{D}}_{\tau_{k}} := T_{{q_{k}}}\pi \left( \mathcal{D}
_{q_{k}}\right)
\]
and
\begin{equation*}
\begin{split}
\delta v_{k} &:= T_{\left( q_{k},w_{k}\right)}\rho \left( h^{q_{k}}\left( \delta
\tau_{k}\right), T_{\left( q_{k},q_{k+1}\right)}\mathcal{A}_{d} \left(
h^{q_{k}}\left( \delta \tau_{k}\right) ,h^{q_{k+1}}\left( \delta \tau_{k+1}\right)
\right) \right) \medskip \\ 
& \ \ \ + T_{\left( q_{k},w_{k}\right)}\rho \left( \left( \xi _{k}\right) _{Q}\left(
q_{k}\right) ,T_{\left( q_{k},q_{k+1}\right)}\mathcal{A}_{d} \left( \left( \xi
_{k}\right) _{Q}\left( q_{k}\right) ,\left( \xi _{k+1}\right) _{Q}\left(
q_{k+1}\right) \right) \right)%
\end{split}
\label{EqDelta_v_k}
\end{equation*}%
where $\left( q_{k},\xi _{k}\right) \in \mathfrak{g}_{q_{k}}^{\mathcal{D}}$.

\item $\left( v_{k},\tau_{k+1}\right) \in \mathcal{\hat{D}}_{d}$ for all $k$ and $\left( v_{\cdot },\tau_{\cdot }\right) $ satisfy the following conditions for each $\left( v_{k-1},\tau_{k},v_{k},\tau_{k+1}\right)$ fixed:

\begin{itemize}
\item $\phi \in T_{\tau_{k}}^{\ast }\left( Q/G\right) $ defined by
\begin{equation*}
\begin{split}
\phi := & D_{1}\check{L}_{d}\left( q_{k},w_{k},\tau_{k+1}\right) \circ
h^{q_{k}}+D_{3}\check{L}_{d}\left( q_{k-1},w_{k-1},\tau_{k}\right) \medskip \\ 
& +\hat{f}_{d}^{-}\left( v_{k},\tau_{k+1}\right) +\hat{f}_{d}^{+}\left(
v_{k-1},\tau_{k}\right)%
\end{split}
\label{EqDefinicion_phi}
\end{equation*}%
vanishes on $\mathcal{\hat{D}}_{\tau_{k}}$, i.e., $\phi \in (\mathcal{\hat{D}}_{\tau_{k}})^\circ$, where
\[
\hat{f}_d(v_0,\tau_1)(\delta v_0,\delta \tau_1) := D_2\check{L}_d(q_0,w_0,\tau_1) T_{(q_0,q_1)}\calA_d(h^{q_0}(\delta \tau_0),h^{q_1}(\delta \tau_1)),
\]
with $v_0 = \rho(q_0,w_0)$ and $q_1 := \tilde{F}_1(q_0,w_0,\tau_1)$.

\item $\psi \in \mathfrak{g}^{\ast }$ defined by
\begin{equation*}
\psi :=D_{2}\check{L}_{d}\left( q_{k-1},w_{k-1},\tau_{k}\right)
w_{k-1}^{-1}-D_{2}\check{L}_{d}\left( q_{k},w_{k},\tau_{k+1}\right) w_{k}^{-1}
\label{EqDefinicion_psi}
\end{equation*}%
vanishes on $\mathfrak{g}_{q_{k}}^{\mathcal{D}}$, i.e., $\psi \in \left( \mathfrak{g}_{q_{k}}^{\mathcal{D}}\right) ^{\circ }$.
\end{itemize}
\end{enumerate}
\end{theorem}

\section{Reduction by stages of symmetric discrete mechanical systems}\label{section:reduction by stages}

In some cases, if $G$ is a symmetry group of a mechanical system, it can be convenient to consider a partial reduction, i.e., the reduction of the system by a subgroup $H \subs G$ and, eventually, as a second step, the reduction by any remaining symmetry of the associated reduced system. This process is called \emph{reduction by (two) stages} and it is natural to wonder if it is equivalent to the reduction by the complete group $G$.

Given a Lagrangian mechanical system, it is well known that its configuration space is described using a tangent bundle. When this system presents a symmetry and a reduction process is applied, one obtains another dynamical system which, in general, is no longer a mechanical system in the Lagrangian sense, because the space in which its dynamics occurs cannot be naturally identified with a tangent bundle. Therefore, the reduction techniques that had been developed for mechanical systems cannot be applied to these reduced systems. A solution to this problem is given, for example, in \cite{ar:cendra_marsden_ratiu:2001:lagrangian_reduction_by_stages}, where they construct a class of dynamical systems and a reduction process for the systems that are symmetric and that is closed in the class; among the elements of this class are the usual mechanical systems as well as their reductions; moreover, when seen as elements of this larger class of systems, the dynamics of both the mechanical systems as well as their reductions are the same as the ones in the standard variational version. The class of systems considered in \cite{ar:cendra_marsden_ratiu:2001:lagrangian_reduction_by_stages} also include the mechanical systems with holonomic constraints and their reductions. In the case of Lagrangian systems with nonholonomic constraints, an analog category has not been defined, but the problem of writing equations of motion for the reduced system obtained after an arbitrary number of stages has been studied, as can be seen in \cite{ar:cendra_diaz:2018:lagrange-dalembert-poincare_equations_by_several_stages}.

In the case of discrete mechanical systems, we have addressed the problem of symmetry reduction by stages in two steps: first for the case of systems without constraints \cite{ar:fernandez_tori_zuccalli:2016:lagrangian_reduction_of_discrete_mechanical_systems_by_stages} and then for the case of nonholonomic systems \cite{ar:fernandez_tori_zuccalli:2020:lagrangian_reduction_of_nonholonomic_discrete_mechanical_systems_by_stages}. In both cases, we have defined categories of systems that are closed under a symmetry reduction process. Below, we will briefly recall the general ideas of the nonholonomic case and see how it allows us to consider the particular case of holonomic constraints.

In \cite{ar:fernandez_tori_zuccalli:2020:lagrangian_reduction_of_nonholonomic_discrete_mechanical_systems_by_stages}, we defined a family of systems --called \emph{discrete Lagrange--d'Alembert--Poincar\'e systems} (DLDPS)-- that naturally includes the symmetric nonholonomic systems and their reductions, in the sense of Section \ref{section:nonholonomic_reduction}.

To recall the construction of these systems, we consider a fiber bundle $\displaystyle {\phi : E \rightarrow M}$ and the manifold $C'(E):=E\times M$ as a fiber bundle over $M$ by $\phi \circ \pr_1$.

Let $\calD$ be a subbundle of $\pr_1^*(TE) \subs T(C'(E))$. A \textit{nonholonomic infinitesimal variation chaining map} (NIVCM) on $(E,\calD)$ is a homomorphism of vector bundles over $\widetilde{\pr_1}$, according to the following commutative diagram 
	\begin{equation*}
		\xymatrix{
			{\mathcal{D}} \ar[d] & \ti{\pr_{34}}^{*}(\mathcal{D}) \ar[l] \ar[d]
			\ar[r]^{\mathcal{P}} & {\ker(d\phi)} \ar[d]  \ar@{^{(}->}[r] &
			{TE} \ar[dl]\\
			{E\times M} & {C''(E)} 
			\ar[l]^{\ti{\pr_{34}}} \ar[r]_{\ti{p_1}} & {E} & {}
		}
	\end{equation*}
	where $\ti{\pr_1} ((\epsilon_0,m_1),(\epsilon_1,m_2)) := \epsilon_0$
	and
	$\ti{\pr_{34}} ((\epsilon_0,m_1),(\epsilon_1,m_2)) :=
	(\epsilon_1,m_2)$.

	A \jdef{discrete Lagrange--D'Alembert--Poincar\'e system} (DLDPS) over $E$ is a
	collection
	$\mathcal{M} := (E,L_d,\mathcal{D}_d,\mathcal{D},\mathcal{P})$ where
	$L_d:C'(E) \rightarrow \R$ is a smooth function, the \jdef{discrete
		Lagrangian}, $\mathcal{D}_d\subs C'(E)$ is a regular
	submanifold, the \jdef{kinematic constraints}, $\mathcal{D}$ is a
	subbundle of $\pr_{1}^{*}(TE)$, the \jdef{variational constraints},
	and $\mathcal{P}$ is a NIVCM over $(E,\mathcal{D})$.

Let us consider the second--order manifold $C''(E) := (E \times M) \times_{\pr_2,\phi \circ \pr_1} (E \times M)$ as a fiber bundle over $M$ with $\tilde{\pr_2} := \pr_2|_{C''(E)}$, where $\pr_2 : E \times M \times E \times M \rightarrow M$ is the canonical projection.

A discrete curve on $C'(E)$ is a set 
$(\epsilon_{.}, m_{.}) = ((\epsilon_{0}, m_1),
\ldots, (\epsilon_{N-1}, m_N))$, where $ ((\epsilon_{k}, m_{k+1}), (\epsilon_{k+1}, m_{k+2})) \in C''(E)$ for $k = 0,\ldots, N-2$.

The \jdef{discrete action} of $\mathcal{M}$ is a function
$S_d :C'(E)^{N} \rightarrow \R$ defined by
$S_d(\epsilon_\cdot,m_\cdot) := \sum_{k=0}^{N-1}
L_{d}(\epsilon_{k},m_{k+1}) $ and a \jdef{trajectory of} $\mathcal{M}$
is a discrete path $(\epsilon_\cdot,m_\cdot) \in C'(E)^{N}$ such
that $(\epsilon_{k},m_{k+1})\in \mathcal{D}_d$ for all
$k=0,\ldots,N-1$ and
\begin{equation*}
	dS_{d}(\epsilon_\cdot,m_\cdot)(\delta\epsilon_\cdot,\delta m_\cdot)=0
\end{equation*}
for all nonholonomic infinitesimal variations
$(\delta\epsilon_\cdot,\delta m_\cdot)$ on
$(\epsilon_\cdot,m_\cdot)$ with fixed endpoints (these variations are defined in terms of the chaining map $\calP$).

Both discrete nonholonomic systems and their reductions
can be seen as DLDPSs and their dynamics (in the sense of the previous paragraph) is the same as the one that they have in the sense of Section 4 (see Example 3.9 and Section 3.2 of \cite{ar:fernandez_tori_zuccalli:2020:lagrangian_reduction_of_nonholonomic_discrete_mechanical_systems_by_stages}).



A convenient notion of morphism between DLDPS is introduced and it is seen that a category $\mathfrak{LDP}_d$, whose objects are the DLDPS, can be defined. In a natural way, it can be verified that the category $\mathfrak{LP}_d$ of discrete Lagrange--Poincar\'e systems defined in \cite{ar:fernandez_tori_zuccalli:2016:lagrangian_reduction_of_discrete_mechanical_systems_by_stages} (with a completely analogous notion of morphism) is a complete subcategory of $\mathfrak{LDP}_d$.

In turn, a notion of symmetry is also considered in a natural way. Roughly speaking, a Lie group $G$ is a symmetry group of a DLDPS if it acts on the underlying fiber bundle preserving the different structures we are considering. This can be seen in detail in \cite[Definition 5.3]{ar:fernandez_tori_zuccalli:2020:lagrangian_reduction_of_nonholonomic_discrete_mechanical_systems_by_stages}. When $G$ is a symmetry group of a DLDPS $\mathcal{M}$, following the ideas of Section 5.2, one can construct a new DLDPS which we denote $\mathcal{M}/G$. We call it ``reduced system'' and it is again a DLDPS in a natural way (see \cite[Definition 5.11]{ar:fernandez_tori_zuccalli:2020:lagrangian_reduction_of_nonholonomic_discrete_mechanical_systems_by_stages}). In fact, the construction requires an additional ingredient: an affine discrete connection $\calA_d$ on a certain principal bundle. We proved the fact that the reduced systems obtained using different affine discrete connections are always isomorphic as DLDPS \cite[Proposition 5.14]{ar:fernandez_tori_zuccalli:2020:lagrangian_reduction_of_nonholonomic_discrete_mechanical_systems_by_stages}. In addition, the reduction map $\Upsilon_{\calA_d} : \mathcal{M} \rightarrow \mathcal{M}/G$ is a morphism in the category $\mathfrak{LDP}_d$, and Corollary 5.16 and Theorem 5.17 prove that the reduction map determines a bijective correspondence between the trajectories of $\mathcal{M}$ and those of $\mathcal{M}/G$.

Finally, we prove that, under certain conditions, the reduction by two stages (first by a closed and normal subgroup of the symmetry group and, then, by the residual symmetry group) produces a system which is isomorphic in $\mathfrak{LDP}_d$ to the system obtained performing a one step reduction by the complete symmetry group. In addition, we prove that there is a diffeomorphism between the trajectories of these systems.

\begin{equation*}
	\xymatrix{
		{} & {} & {\mathcal{M}} \ar[ddll]_{\Upsilon_{\mathcal{A}_d^G}} 
		\ar[dr]^{\Upsilon_{\mathcal{A}_d^H}} & {} & {}\\
		{} & {} & {} & {\mathcal{M}^H} \ar[dr]^{\Upsilon_{\mathcal{A}_d^{G/H}}} & {}\\
		{\mathcal{M}^{G}} \ar[rrrr]^{\simeq} & {} & {} & {} & {\mathcal{M}^{G/H}}}
\end{equation*}


\section*{Acknowledgments}
This document is the result of research partially supported by grants
from the Universidad Nacional de Cuyo [code 06/80020240100069UN],
Universidad Nacional de La Plata [code SX007] and CONICET.


\printbibliography

@article{ar:balseiro_solomin:2008:on_generalized_non-holonomic_systems,
	author = {P. {Balseiro} and J. {Solomin}},
	title = {On generalized	non-holonomic systems},
	journal = {Lett. Math. Phys.},
	fjournal = {Letters in Mathematical Physics},
	volume = {84},
	number = {1},
	pages = {15--30},
	year = {2008},
}

@article{ar:borda_fernandez_grillo:2013:discrete_second_order_constrained_lagrangian_systems_first_results,
  author = {Nicol{\'a}s Borda and Javier {Fern{\'a}ndez} and Sergio {Grillo}},
  title = {Discrete second order constrained {L}agrangian systems: first results},
  journal = {J. Geom. Mech.},
  fjournal = {Journal of Geometric Mechanics},
  year = {2013},
  volume = {5},
  number = {4},
  pages = {381--397},
  note = {Also, \href{http://arXiv.org/abs/1312.1941}{{\tt arXiv:1312.1941}}},
}

@article{ar:caruso_fernandez_tori_zuccalli:2023:lagrangian_reduction_of_forced_discrete_mechanical_systems,
author = {Matías I {Caruso} and Javier {Fernández} and Cora {Tori} and Marcela {Zuccalli}},
title = {Lagrangian reduction of forced discrete mechanical systems},
journal = {J. Phys. A, Math. Theor.},
fjournal = {Journal of Physics A: Mathematical and Theoretical},
doi = {10.1088/1751-8121/aceae3},
year = {2023},
publisher = {IOP Publishing},
volume = {56},
number = {35},
pages = {355202},
}

@article{ar:caruso_fernandez_tori_zuccalli:2026:remarks_on_structures_and_preservation_in_forced_discrete_mechanical_systems_of_Routh_type,
  author = {Matías I {Caruso} and Javier {Fernández} and Cora {Tori} and Marcela {Zuccalli}},
  title = {Remarks on structures and preservation in forced discrete mechanical systems of Routh type},
  journal = {J. Geom. Phys.},
  fjournal = {Journal of Geometry and Physics},
  volume = {223},
  pages = {105776},
  year = {2026},
  issn = {0393-0440},
  doi = {https://doi.org/10.1016/j.geomphys.2026.105776},
}

@article{ar:cendra_diaz:2018:lagrange-dalembert-poincare_equations_by_several_stages,
	author = {H. {Cendra} and V. {D\'iaz}},
	title = {Lagrange--D'Alembert--Poincar\`{e} Equations by Several Stages},
	journal = {J. Geom. Mech.},
	fjournal = {Journal of Geometric Mechanics},
	volume = {10},
	number = {1},
	pages = {1--41},
	year = {2018},
}

@article {ar:cendra_grillo:2006:generalized_nonholonomic_mechanics_servomechanisms_and_related_brackets,
    AUTHOR = {H. {Cendra} and S. {Grillo}},
     TITLE = {Generalized nonholonomic mechanics, servomechanisms and
              related brackets},
   JOURNAL = {J. Math. Phys.},
  FJOURNAL = {Journal of Mathematical Physics},
    VOLUME = {47},
      YEAR = {2006},
    NUMBER = {2},
     PAGES = {022902, 29},
      ISSN = {0022-2488},
     CODEN = {JMAPAQ},
   MRCLASS = {70F25 (70Q05 93B52)},
  MRNUMBER = {MR2208156 (2007a:70019)},
MRREVIEWER = {Ram Krishan Sharma},
}

@article{ar:cendra_marsden_ratiu:2001:lagrangian_reduction_by_stages,
    AUTHOR = {Hern{\'a}n {Cendra} and Jerrold E. {Marsden} and Tudor S. {Ratiu}},
     TITLE = {Lagrangian reduction by stages},
   JOURNAL = {Mem. Amer. Math. Soc.},
  FJOURNAL = {Memoirs of the American Mathematical Society},
    VOLUME = {152},
      YEAR = {2001},
    NUMBER = {722},
     PAGES = {x+108},
      ISSN = {0065-9266},
     CODEN = {MAMCAU},
   MRCLASS = {37J15 (53D20 70H33)},
  MRNUMBER = {MR1840979 (2002c:37081)},
MRREVIEWER = {Marco Castrillon Lopez},
}

@incollection {ar:cendra_marsden_ratiu:2001:geometric_mechanics_lagrangian_reduction_and_nonholonomic_systems,
    AUTHOR = {Hern{\'a}n {Cendra} and Jerrold E. {Marsden} and Tudor S. {Ratiu}},
     TITLE = {Geometric mechanics, {L}agrangian reduction, and nonholonomic
              systems},
 BOOKTITLE = {Mathematics unlimited---2001 and beyond},
     PAGES = {221--273},
 PUBLISHER = {Springer},
   ADDRESS = {Berlin},
      YEAR = {2001},
   MRCLASS = {37J15 (37J60 53D20 70F25 70H33)},
  MRNUMBER = {MR1852159 (2002g:37067)},
MRREVIEWER = {Charles-Michel Marle},
}

@article {ar:cortes_martinez:2001:non_holonomic_integrators,
    AUTHOR = {J. {Cortés} and S. {Martínez}},
     TITLE = {Non-holonomic integrators},
   JOURNAL = {Nonlinearity},
  FJOURNAL = {Nonlinearity},
    VOLUME = {14},
    NUMBER = {5},
      YEAR = {2001},
     PAGES = {1365--1392},
      ISSN = {0951-7715},
     CODEN = {NONLE5},
   MRCLASS = {37M15 (37J60 65P10 70-08 70F25)},
  MRNUMBER = {MR1862825 (2002h:37165)},
MRREVIEWER = {Antonella Zanna},
}

@article{ar:fernandez_juchani_zuccalli:2022:discrete_connections_on_principal_bundles_the_discrete_atiyah_sequence,
	author = {Javier {Fern{\'a}ndez} and Mariana {Juchani} and Marcela {Zuccalli}},
	title = {Discrete connections on principal bundles: The discrete Atiyah sequence},
	journal = {J. Geom. Phys.},
	fjournal = {Journal of Geometry and Physics},
	volume = {172},
	number = {27},
	pages = {104417},
	year = {2022},
}

@article{ar:fernandez_kordon:2025:the_integration_problem_for_principal_connections,
	author = {Javier {Fern{\'a}ndez} and Francisco {Kordon}},
	title = {The Integration Problem for principal connections},
	journal = {J. Geom. Phys.},
	fjournal = {Journal of Geometry and Physics},
	volume = {216},
	number = {20},
	pages = {105566},
	year = {2025},
}

@article {ar:fernandez_tori_zuccalli:2010:lagrangian_reduction_of_discrete_mechanical_systems,
    AUTHOR = {Javier {Fern{\'a}ndez} and Cora {Tori} and Marcela {Zuccalli}},
     TITLE = {Lagrangian reduction of nonholonomic discrete mechanical
              systems},
   JOURNAL = {J. Geom. Mech.},
  FJOURNAL = {Journal of Geometric Mechanics},
    VOLUME = {2},
      YEAR = {2010},
    NUMBER = {1},
     PAGES = {69--111},
      ISSN = {1941-4889},
   MRCLASS = {37J60 (37J15 70F25 70Hxx)},
  MRNUMBER = {2646536},
       DOI = {10.3934/jgm.2010.2.69},
       URL = {http://dx.doi.org/10.3934/jgm.2010.2.69},
      note = {Also, \href{http://arXiv.org/abs/1004.4288}{{\tt arXiv:1004.4288}}},
}

@article {ar:fernandez_tori_zuccalli:2016:lagrangian_reduction_of_discrete_mechanical_systems_by_stages,
    AUTHOR = {Javier {Fern{\'a}ndez} and Cora {Tori} and Marcela {Zuccalli}},
     TITLE = {Lagrangian reduction of discrete mechanical systems by stages},
   JOURNAL = {J. Geom. Mech.},
  FJOURNAL = {Journal of Geometric Mechanics},
    VOLUME = {8},
      YEAR = {2016},
    NUMBER = {1},
     PAGES = {35--70},
      ISSN = {1941-4889},
   MRCLASS = {Preliminary Data},
  MRNUMBER = {3485921},
       DOI = {10.3934/jgm.2016.8.35},
       URL = {http://dx.doi.org/10.3934/jgm.2016.8.35},
       note = {Also, \href{http://arXiv.org/abs/1511.06682}{{\tt
                  arXiv:1511.06682 [math.DG]}}},
}

@article {ar:fernandez_tori_zuccalli:2020:lagrangian_reduction_of_nonholonomic_discrete_mechanical_systems_by_stages,
	AUTHOR = {Javier {Fern{\'a}ndez} and Cora {Tori} and Marcela {Zuccalli}},
	TITLE = {Lagrangian reduction of nonholonomic discrete mechanical systems by stages},
	JOURNAL = {J. Geom. Mech.},
	FJOURNAL = {Journal of Geometric Mechanics},
	VOLUME = {12},
	YEAR = {2020},
	NUMBER = {4},
	PAGES = {607-639},
	ISSN = {1941-4889},
	DOI = {10.3934/jgm.2020029},
}

@article{ar:fernandez_zuccalli:2013:a_geometric_approach_to_discrete_connections_on_principal_bundles,
	author = {Javier {Fern{\'a}ndez} and Marcela {Zuccalli}},
	title = {A geometric approach to discrete connections on principal bundles},
	journal = {J. Geom. Mech.},
	fjournal = {Journal of Geometric Mechanics},
	volume = {5},
	issue = {4},
	pages = {433--444},
	year = {2013},
}

@article {ar:jalnapurkar_leok_marsden_west:2006:discrete_routh_reduction,
    AUTHOR = {Sameer M. {Jalnapurkar} and Melvin {Leok} and Jerrold
              E. {Marsden} and Matthew {West}},
     TITLE = {Discrete {R}outh reduction},
   JOURNAL = {J. Phys. A},
  FJOURNAL = {Journal of Physics. A. Mathematical and General},
    VOLUME = {39},
      YEAR = {2006},
    NUMBER = {19},
     PAGES = {5521--5544},
      ISSN = {0305-4470},
     CODEN = {JPHAC5},
   MRCLASS = {37J15 (37M15 65L06 70-08 70H33)},
  MRNUMBER = {MR2220774 (2007g:37038)},
MRREVIEWER = {Juan Carlos Marrero Gonzalez},
}

@unpublished{un:leok_marsden_weinstein:2005:a_discrete_theory_of_connections_on_principal_bundles,
  author = 	 {Melvin {Leok} and Jerrold E. {Marsden} and Alan {Weinstein}},
  title = 	 {A Discrete Theory of Connections on Principal Bundles},
  note = 	 {\href{http://arXiv.org/abs/math/0508338}{\tt arXiv:math/0508338}},
  OPTkey = 	 {},
  OPTmonth = 	 {},
  year = 	 {2005},
  OPTannote = 	 {}
}

@article {ar:marrero_martin_martinez-discrete_lagrangian_and_hamiltonian_mechanics_on_lie_groupoids-corrigendum,
    AUTHOR = {Marrero, J. C. and Mart{\'{\i}}n de Diego, D. and
              Mart{\'{\i}}nez, E.},
     TITLE = {Corrigendum: ``{D}iscrete {L}agrangian and {H}amiltonian
              mechanics on {L}ie groupoids'' [{N}onlinearity {\bf 19}
              (2006), no. 6, 1313--1348]},
   JOURNAL = {Nonlinearity},
  FJOURNAL = {Nonlinearity},
    VOLUME = {19},
      YEAR = {2006},
    NUMBER = {12},
     PAGES = {3003--3004},
      ISSN = {0951-7715},
     CODEN = {NONLE5},
   MRCLASS = {37J05 (17B66 22A22 53D20 70G45 70H30)},
  MRNUMBER = {2275510 (2007j:37085)},
       DOI = {10.1088/0951-7715/19/12/C01},
       URL = {http://dx.doi.org/10.1088/0951-7715/19/12/C01},
}

@article {ar:marrero_martin_martinez-discrete_lagrangian_and_hamiltonian_mechanics_on_lie_groupoids,
    AUTHOR = {Marrero, Juan C. and Mart{\'{\i}}n de Diego, David and
              Mart{\'{\i}}nez, Eduardo},
     TITLE = {Discrete {L}agrangian and {H}amiltonian mechanics on {L}ie
              groupoids},
   JOURNAL = {Nonlinearity},
  FJOURNAL = {Nonlinearity},
    VOLUME = {19},
      YEAR = {2006},
    NUMBER = {6},
     PAGES = {1313--1348},
      ISSN = {0951-7715},
     CODEN = {NONLE5},
   MRCLASS = {37J05 (17B66 22A22 53D20 70G45 70H30)},
  MRNUMBER = {2230001 (2007c:37068)},
MRREVIEWER = {Pawe{\l}Urbanski},
       DOI = {10.1088/0951-7715/19/6/006},
       URL = {http://dx.doi.org/10.1088/0951-7715/19/6/006},
}

@article {ar:marsden_west:2001:discrete_mechanics_and_variational_integrators,
    AUTHOR = {Jerrold E. {Marsden} and M. {West}},
     TITLE = {Discrete mechanics and variational integrators},
   JOURNAL = {Acta Numer.},
  FJOURNAL = {Acta Numerica},
    VOLUME = {10},
      YEAR = {2001},
     PAGES = {357--514},
      ISSN = {0962-4929},
   MRCLASS = {37M15 (37J05 65P10 70-08 70H05)},
  MRNUMBER = {MR2009697 (2004h:37130)},
MRREVIEWER = {Christian Lubich},
}

@article {ar:mclachlan_perlmutter:2006:integrators_for_nonholonomic_mechanical_systems,
    AUTHOR = {R. {McLachlan} and M. {Perlmutter}},
     TITLE = {Integrators for nonholonomic mechanical systems},
   JOURNAL = {J. Nonlinear Sci.},
  FJOURNAL = {Journal of Nonlinear Science},
    VOLUME = {16},
      YEAR = {2006},
    NUMBER = {4},
     PAGES = {283--328},
      ISSN = {0938-8974},
   MRCLASS = {37M15 (37J60 65P10 70-08 70F25)},
  MRNUMBER = {MR2254707 (2008d:37154)},
MRREVIEWER = {David Mart{\'{\i}}n de Diego},
}

@article {ar:rodriguezAbella_leok-discrete_dirac_reduction_of_implicit_lagrangian_systems_with_abelian_symmetry_groups,
    AUTHOR = {Rodr\'{\i}guez Abella, \'{A}lvaro and Leok, Melvin},
     TITLE = {Discrete {D}irac reduction of implicit {L}agrangian systems
              with abelian symmetry groups},
   JOURNAL = {J. Geom. Mech.},
  FJOURNAL = {Journal of Geometric Mechanics},
    VOLUME = {15},
      YEAR = {2023},
    NUMBER = {1},
     PAGES = {319--356},
      ISSN = {1941-4889,1941-4897},
   MRCLASS = {37J39 (65P10 70G65 70H33)},
  MRNUMBER = {4579810},
       DOI = {10.3934/jgm.2023013},
       URL = {https://doi.org/10.3934/jgm.2023013},
}

\end{document}